\documentclass[a4paper]{amsart}
        \usepackage{latexsym}
        \usepackage{amssymb}
        \usepackage{amsmath}
        \usepackage{amsfonts}
        \usepackage{amsthm}
        \usepackage[hypertexnames=false]{hyperref}
        \usepackage[all,knot,poly,2cell]{xy}
             \UseAllTwocells
        \usepackage{xspace}

        \usepackage{eucal}
        \usepackage{mathrsfs}

        \theoremstyle{plain}
        \newtheorem{thm}{Theorem}[section]
        \newtheorem{cor}[thm]{Corollary}
        \newtheorem{lem}[thm]{Lemma}
        \newtheorem{prop}[thm]{Proposition}
        \newtheorem{mainthms}{Theorem}

        \theoremstyle{definition}
        \newtheorem{defn}[thm]{Definition}
        \newtheorem{ex}[thm]{Example}
 
        \theoremstyle{remark}
        \newtheorem{rem}[thm]{Remark}

        \newcommand{\suchthat}{\,:\,}
        \newcommand{\itemref}[1]{\eqref{#1}}
        \newcommand{\opcit}[1][]{[\emph{op.\ cit.}{#1}]\xspace}
        \newcommand{\loccit}{[\emph{loc.\ cit.}]\xspace}

        \numberwithin{equation}{section}

        \newcommand{\C}{\mathbb{C}}
        \newcommand{\Z}{\mathbb{Z}}

        \newcommand{\Orb}{\mathcal{O}}    
       \DeclareMathOperator{\spec}{Spec}
        \newcommand{\red}[1]{{#1}_{\mathrm{red}}}
        \newcommand{\COH}[1]{\mathbf{Coh}({#1})}
        \newcommand{\QCOHB}{\mathbf{QCoh}}
        \newcommand{\QCOH}[2][]{\QCOHB^{#1}({#2})}
        \newcommand{\Aff}{\mathbb{A}}
           
           \newcommand{\Et}{\mathrm{\acute{E}t}}
           \newcommand{\bflat}{\mathrm{flb}}
           
           \newcommand{\blfp}{\mathrm{lfpb}}
           \newcommand{\qf}{\mathrm{qf}}

           \newcommand{\fp}{\mathrm{fp}}
           
           \newcommand{\fppf}{\mathrm{fppf}}

           \newcommand{\bcompact}{\mathrm{prb}}

        \newcommand{\QCOHSTK}[3][]{\underline{\mathrm{QCoh}}_{{#2}/{#3}}^{{#1}}}
        
        \newcommand{\Homstk}{\underline{\Hom}}
        
        \newcommand{\Quotshf}{\underline{\mathrm{Quot}}}

        \DeclareMathOperator{\Der}{Der}
        \DeclareMathOperator{\Exal}{Exal}   
        \newcommand{\MOD}[1]{\mathbf{Mod}({#1})}          
        \newcommand{\ideal}{\triangleleft}      

        \DeclareMathOperator{\Def}{Def}
        \newcommand{\EXAL}{{\mathbf{Exal}}}
        \DeclareMathOperator{\Obs}{Obs}
        \DeclareMathOperator{\obs}{obs}
        \newcommand{\DEF}{{\mathbf{Def}}}
        \newcommand{\extn}[1]{[{#1}]}

        \DeclareMathOperator{\coker}{coker}
        
        \DeclareMathOperator{\Hom}{Hom}
        \DeclareMathOperator{\Ext}{Ext}
        \DeclareMathOperator{\Tor}{Tor}

        \newcommand{\STor}{\mathscr{T}or}

        \newcommand{\SHom}{\mathscr{H}om}

        \newcommand{\ID}[1]{\mathrm{Id}_{#1}}    
        \DeclareMathOperator{\Aut}{Aut}
        
        \DeclareMathOperator{\Mor}{Mor}
        
        \DeclareMathOperator{\Ret}{Ret}
        \newcommand{\tensor}{\otimes}

        \newcommand{\opp}{\circ}
        \newcommand{\FIB}[2]{{#1}({#2})}
                \newcommand{\AB}{\mathbf{Ab}}
                \newcommand{\SETS}{\mathbf{Sets}}        
                \newcommand{\Sch}{\mathbf{Sch}}
                \newcommand{\SCH}[2][]{\Sch^{#1}/{#2}}

\newcommand{\Coll}{\mathscr{C}}

\newcommand{\Homog}[1]{\mathbf{{#1}}}
\newcommand{\HNIL}{\Homog{Nil}}
\newcommand{\HCL}{\Homog{Cl}}
\newcommand{\HrNIL}{\Homog{rNil}}
\newcommand{\HrCL}{\Homog{rCl}}

\newcommand{\HA}{\Homog{Aff}}
\newcommand{\trvext}[2]{{\imath}_{{#1},{#2}}}
\newcommand{\trvret}[2]{{r}_{{#1},{#2}}}

\newcommand{\shv}[1]{\mathcal{{#1}}}
\newcommand{\stk}[1]{{{#1}}}
\newcommand{\fndefn}[1]{\emph{{#1}}}

\newcommand{\HS}[2]{\underline{\mathrm{HS}}_{{#1}/{#2}}}
\newcommand{\MOR}[3][]{\underline{\Mor}^{{#1}}_{{#2}/{#3}}} 
\DeclareMathOperator{\Isom}{Isom}
\newcommand{\Isomstk}{\underline{\Isom}}

\newcommand{\COHSTK}[2]{\underline{\mathrm{Coh}}_{{#1}/{#2}}}
\renewcommand{\subset}{\subseteq}
\numberwithin{equation}{section}

\newenvironment{mydescription}{%
   
   \begin{description}%
}{%
   \end{description}%
}       
\newcommand{\spref}[1]{\href{http://stacks.math.columbia.edu/tag/#1}{#1}}
\title{Openness of versality via coherent functors} 
\author{Jack Hall}
\date{April 4th, 2013}
\begin{document}
\begin{abstract} 
We give a proof of openness of versality using coherent functors. As an application, we streamline Artin's criterion for algebraicity of a stack. We also introduce multi-step obstruction theories, employing them to produce obstruction theories for the stack of coherent sheaves, the Quot functor, and spaces of maps in the presence of non-flatness.
\end{abstract}
\address{Department of Mathematics, KTH, 100 44 Stockholm, Sweden}
\email{jackhall@kth.se}
\maketitle
\section*{Introduction}
In M. Artin's classic paper on stacks, a criterion for algebraicity is
expounded \cite[Thm.~5.3]{MR0399094}. In this paper, we take a novel
approach to algebraicity, proving an algebraicity criterion for stacks
which is easier to apply, more widely applicable, and admitting a
substantially simpler proof.
\begin{mainthms}\label{mainthms:repcrit2}
  Fix an excellent scheme $S$ and a category $\stk{X}$, fibered in groupoids
  over the category of $S$-schemes, $\SCH{S}$. Then, $\stk{X}$ is an
  algebraic stack, locally of finite presentation over $S$, if and  only if
  the following conditions are satisfied.   
  \begin{enumerate}
  \item \emph{[Stack]} $\stk{X}$ is a stack over the site
    $(\SCH{S})_\Et$.
  \item \emph{[Limit preservation]} For any inverse system of affine
    $S$-schemes $\{\spec A_j\}_{j\in 
      J}$ with limit $\spec A$, the natural functor: 
    \[
    \varinjlim_j \FIB{\stk{X}}{\spec A_j} \to \FIB{\stk{X}}{\spec A}
    \]
    is an equivalence of categories. 
  \item \emph{[Homogeneity]} For any diagram of affine $S$-schemes
    $[\spec B
    \leftarrow \spec A \xrightarrow{{i}} \spec A']$, with ${i}$
    a nilpotent closed immersion, the natural functor:   
    \[
    \FIB{\stk{X}}{\spec(B\times_{A} A')}
    \to \FIB{\stk{X}}{\spec A'} \times_{\FIB{\stk{X}}{\spec A}}
    \FIB{\stk{X}}{\spec B} 
    \]
    is an equivalence of categories. 
  \item \emph{[Effectivity]} For any local noetherian
    ring $(B,\mathfrak{m})$, such that the ring $B$ is
    $\mathfrak{m}$-adically complete, with an $S$-scheme structure
    $\spec B \to S$ such that the 
    induced morphism $\spec 
    (B/\mathfrak{m}) \to S$ is locally of finite type, 
    the natural functor:
    \[
    \FIB{\stk{X}}{\spec B} \to \varprojlim_n \FIB{\stk{X}}{\spec
      (B/\mathfrak{m}^{n})} 
    \]
    is an equivalence of categories.
  \item \emph{[Conditions on automorphisms and deformations]} For any affine
    $S$-scheme $T$, locally of finite type 
    over $S$, and $\xi\in \stk{X}(T)$, the functors $\Aut_{\stk{X}/S}(\xi, -)$,
    $\Def_{\stk{X}/S}(\xi,-) : \QCOH{T} \to \AB$ are coherent. 
  \item \emph{[Conditions on obstructions]} For any affine $S$-scheme $T$,
    locally of finite type over $S$, and $\xi\in \stk{X}(T)$, there exists
    an integer $n$ and a 
    coherent $n$-step obstruction theory for $\stk{X}$ at $\xi$.  
  \end{enumerate}
\end{mainthms}
Except for conditions (5) and (6), Theorem
\ref{mainthms:repcrit2} is similar to Artin's criterion \cite[Thm.\
5.3]{MR0399094}. Note, however, that we have fewer conditions, and
these conditions are cleaner (e.g.\ no deformation situations).  The
conditions of Theorem \ref{mainthms:repcrit2} are also stable under
composition, in the sense of \cite{starr-2006}.      

This paper began with the realization that the homogeneity condition
(3), which is stronger than the analogous condition of
\cite[(S1')]{MR0399094}, together with conditions (5) and (6), simplifies and broadens the
applicability of existing results.

Our usage of the term ``coherent'' in conditions (5) and (6) of
Theorem \ref{mainthms:repcrit2} is in a different sense than what many
readers may be familiar with, so we recall the following definition of
M. Auslander \cite{MR0212070}. For an affine scheme $S$, a functor $F
: \QCOH{S} \to \AB$ is \fndefn{coherent} if there exists a morphism of
quasicoherent $\Orb_S$-modules $\shv{K}_1 \to \shv{K}_2$, such that
for all $\shv{I}\in \QCOH{S}$, there is a natural isomorphism of
abelian groups:
\[
F(\shv{I}) \cong \coker(\Hom_{\Orb_S}(\shv{K}_2,\shv{I}) \to
\Hom_{\Orb_S}(\shv{K}_1,\shv{I})). 
\]
It is proven in \cite{hallj_coho_bc} that most functors arising in
moduli are coherent.
\subsection*{Relation with other work}
The idea of using the $\Exal$ functors to simplify M. Artin's results
should be attributed to H. Flenner \cite{MR638811}. Our results and
techniques are quite different, however.  In particular, H. Flenner
\opcit does not address the relationship between formal smoothness and
formal versality. 

Independently, 
work in the Stacks Project \cite[\spref{07T0}]{stacks-project} has 
provided a different perspective on Artin's results. This approach, however, requires that the deformation--obstruction theory is given by
a bounded complex. If there are non-flat or non-tame objects in the
moduli problem, the existence of such a complex is subtle. The
problems with non-tame stacks can be dealt with by
\cite[Thm.~B]{hallj_coho_bc}. The problems with non-flatness can be
handled by derived algebraic geometry
\cite[\href{http://math.columbia.edu/~dejong/wordpress/?p=2572}{blog:2572}]{stacks-project}
or $2$-step obstruction theories
(c.f.~\S\S\ref{sec:example_qcoh}--\ref{sec:example_hs}). 

Using the ideas of B. T\"oen and G. Vezzosi
\cite[1.4]{MR2394633}, J. Lurie has developed a criterion for
algebraicity in the derived context \cite[Thm.\
3.2.1]{dag14}. Conditions (5) and (6) of Theorem
\ref{mainthms:repcrit2} are related 
to Lurie's requirement of the existence of a cotangent complex. As
Lurie observes, his criterion is not applicable to Artin
stacks, though it is a future intention to make it so \cite[Rem.\
2]{dag14}. J. Pridham has proved a criterion for Artin stacks
\cite[Thm.\ 3.16]{pridham_crit}, which is related to the results of
Lurie's PhD Thesis \cite[Thm.\ 7.1.6 \& Thm.\ 7.5.1]{luriethesis}.

To prove that the Quot functors for
separated Deligne-Mumford stacks are algebraic spaces, M. Olsson and
J. Starr \cite[Thm.~1.1]{MR2007396} did not apply \cite[Cor.\
5.4]{MR0399094}, which like \opcit[, Thm.\ 5.3], is formulated in
terms of a single-step obstruction theory. The reason for this is
simple: in the presence of non-flatness, it is difficult to formulate
a single-step obstruction theory with good properties. 

They circumvented this predicament by the use of Artin's original
algebraicity criterion \cite[Thm.\ 5.3]{MR0260746}. This earlier
algebraicity criterion is not formulated in terms of the existence and
properties of a single-step obstruction theory, but in terms of
certain explicit lifting problems---making its application more
complicated (note that J. Starr \cite[Thm.\ 2.15]{starr-2006} has
subsequently generalized the criteria of \cite[Thm.\ 5.3]{MR0260746}
to stacks).  To solve these lifting problems, M. Olsson and J. Starr
\cite[Lem.\ 2.5]{MR2007396} used a $2$-step process. This $2$-step
process is insufficiently functorial to define a multi-step
obstruction theory in the sense of this paper. It is, however, closely
related, and inspired the multi-step obstruction theories we define.

M. Olsson and J. Starr \cite[p.\ 4077]{MR2007396} noted that M. Artin
had incorrectly computed the obstruction theory of the Quot functor in
the prescence of non-flatness \cite[6.4]{MR0260746}. We have also
located some other articles in the literature that have not observed
the subtlety of deformation theory in the presence of non-flatness
(see \S\ref{sec:example_qcoh} and \S\ref{sec:example_hs}). We would
like to emphasize that the impact of this on the main ideas of these
articles is nil. Indeed, the relevant arguments in these articles are
still perfectly valid in the flat case---covering most cases of
interest to geometers. In the non-flat case, the relevant statements
in these articles can be shown to hold with the techniques and
examples of this article.

By work of M. Olsson \cite[Rem.\
1.7]{MR2206635}, the conditions of Theorem 
\ref{mainthms:repcrit2} are seen to be necessary. The sufficiency of
the conditions of Theorem \ref{mainthms:repcrit2} is
demonstrated by the following sequence of observations:   
\begin{enumerate}
\item[(i)] the existence of formally versal deformations,
\item[(ii)] the existence of algebraizations of formally versal
  deformations, and 
\item[(iii)] formal versality at a point implies smoothness in a
  neighbourhood. 
\end{enumerate}
Using the generalizations of M. Artin's techniques \cite{MR0399094}
due to B. Conrad and J. de Jong \cite[Thm.\ 1.5]{MR1935511},
conditions (1)--(4) of Theorem \ref{mainthms:repcrit2} prove (i) and
(ii). The main contribution of this paper is the usage of conditions
(3), (5), and (6) of Theorem \ref{mainthms:repcrit2} to prove
(iii). 

Note that in our proof of (iii), the techniques of Artin
approximation \cite{MR0268188} are not used. This is in
contrast to M. Artin's treatments \cite{MR0260746,MR0399094}, where
this technique features prominently. In a paper joint with D. Rydh
\cite{hallj_dary_artin_crit}, we illustrate how refinements of the
homogeneity condition (3) clarify and simplify M. Artin's results on
versality.  
\subsection*{Outline}
In \S\ref{sec:homog}, we discuss the notion of
homogeneity. Homogeneity is a generalization of the Schlessinger-Rim
criteria \cite[Exp.\ VI]{SGA7}. This section is quite categorical, but
it is the only section of the paper that is such. Morally, homogeneity
provides a stack $\stk{X}$ with a linear structure at every point,
which we describe in \S\ref{sec:extensions}. To be precise, for any
scheme $T$, together with an object $\xi \in \FIB{X}{T}$, homogeneity
produces an additive functor $\Exal_{\stk{X}}(\xi,-) : \QCOH{T} \to
\AB$ sharply controlling the deformation theory of $\xi$. The author
learnt these ideas from J. Wise (in person) and his paper
\cite{2011arXiv1111.4200W}, though they are likely well-known, and go
back at least as far as the work of H. Flenner \cite{MR638811}. In
\S\ref{sec:limit_preservation}, we recall and generalize---to the
relative setting---the notion of limit preserving groupoid
\cite[\S1]{MR0399094}.

In \S\ref{sec:fs_fv}, we recall the notions of formal versality and formal 
smoothness. Then, we recast these notions in terms of vanishing criteria
for the functors $\Exal_{\stk{X}}(T,-)$. The central
technical result of this paper is Theorem
\ref{thm:opennessversal}---our new proof of (iii).   

In \S\ref{sec:coherent}, we briefly review coherent functors. In
\S\ref{sec:def_obs_comp}, we formalize multi-step obstruction
theories. In \S\ref{sec:crit_alg}, we prove Theorem 
\ref{mainthms:repcrit2}. 

The remainder of the paper is devoted to applications. In
\S\ref{sec:example_qcoh}, we compute a $2$-step obstruction theory for
the stack of coherent sheaves. Finally, in \S\ref{sec:example_hs}, we
compute a $2$-step obstruction theory
for the stack of morphisms between two algebraic stacks.

In Appendix \ref{app:schlessinger_pushouts}, we prove that pushouts of
algebraic stacks along nilimmersions and 
affine morphisms exist. This enables the verification of the homogeneity
condition (3) in practice. In Appendix \ref{app:tor}, we state two
basic results on local $\Tor$-functors for morphisms of algebraic
stacks. 
\subsection*{Assumptions, conventions, and notations}
For a category $\mathscr{C}$, denote the opposite category by
$\mathscr{C}^\opp$. A \fndefn{fibration} of categories $Q :
\mathscr{C} \to \mathscr{D}$ has the property that every arrow in the
category $\mathscr{D}$ admits a strongly cartesian lift. For an object
$d$ of the category $\mathscr{D}$, we denote the resulting fiber
category by $\FIB{Q}{d}$. It will also be convenient to say that the
category $\mathscr{C}$ is \fndefn{fibered} over $\mathscr{D}$. If the
category $\mathscr{C}$ is fibered over $\mathscr{D}$, and every arrow
in the category $\mathscr{C}$ is strongly cartesian, then we say that
the functor $Q$ is fibered in groupoids. The assumptions guarantee
that if the category $\mathscr{C}$ is fibered in groupoids over
$\mathscr{D}$, then for every object $d$ of the category
$\mathscr{D}$, the fiber category $\FIB{Q}{d}$ is a groupoid.

For a scheme $T$, denote by $|T|$ the underlying topological space
(with the Zariski topology) and $\Orb_T$ the (Zariski) sheaf of rings
on $|T|$. For $t\in |T|$, let $\kappa(t)$ denote the residue field.
Denote by $\QCOH{T}$ (resp.\ $\COH{T}$) the abelian category of
quasicoherent (resp.\ coherent) sheaves on the scheme $T$. Let
$\SCH{T}$ denote the category of schemes over $T$. The big \'etale
site over $T$ will be denoted by $(\SCH{T})_\Et$.  

For a ring $A$ and an $A$-module $M$, denote the quasicoherent
$\Orb_{\spec A}$-module associated to $M$ by $\widetilde{M}$. Denote
the abelian category of all (resp.\ coherent) $A$-modules by $\MOD{A}$
(resp.\ $\COH{A}$). 

As in \cite{stacks-project}, we make no separation assumptions on our
algebraic stacks and spaces. As in \cite{MR2312554}, we use the
lisse-\'etale site for sheaves on algebraic stacks. 

Fix a $1$-morphism of algebraic stacks $f : X \to Y$. Given another
$1$-morphism of algebraic stacks $W \to Y$ we denote the pullback along this
$1$-morphism by $f_W : X_W \to W$. 

A morphism of algebraic $S$-stacks $U \to V$ is a \fndefn{locally
  nilpotent closed immersion} if it is a closed immersion defined by a
quasicoherent sheaf of ideals $\shv{I}$, such that $\fppf$-locally on
$V$ there always exists an integer $n$ such that $\shv{I}^n = (0)$.
\subsection*{Acknowledgements}  
I would like to thank R.~Ile, R.~Skjelnes, and
B.~Williams for  some interesting conversations. A special thanks goes to 
J.~Wise for explaining to me the notion of homogeneity. A very special
thanks is due to D.~Rydh for his tremendous patience and enthusiasm.
\section{Homogeneity}\label{sec:homog}
Schlessinger's conditions \cite{MR0217093}, for a functor of artinian
rings, are fundamental to the theory and understanding of infinitesimal
deformation theory. This was generalized to groupoids by  R.S. Rim
\cite[Exp.\ VI]{SGA7}, clarifying infinitesimal deformation
theory in the presence of automorphisms. These conditions are
instances of the notion of homogeneity, which can be  
traced back to A. Grothendieck \cite[195.II]{MR0146040}. More
recently, a generalisation of these conditions
\cite[Exp.\ VI]{SGA7} was considered by J. Wise
\cite[\S2]{2011arXiv1111.4200W}. In this section, we will develop a 
relative formulation of homogeneity for use in this paper.      

Fix a scheme $S$. An $S$-\fndefn{groupoid} is a pair $(X,a_X)$
consisting of a category $X$ and a fibration in groupoids $a_X : X \to
\SCH{S}$. A $1$-morphism of   
$S$-groupoids $\Phi :(Y,a_Y) \to (Z,a_Z)$ is a functor $\Phi : Y \to Z$ that
commutes strictly over $\SCH{S}$. We will typically refer to an
$S$-groupoid $(X,a_X)$ just as ``$X$''.
\begin{ex}\label{ex:Sch_mor}
  For any $S$-scheme $T$, there is a canonical functor $\SCH{T} \to
  \SCH{S} : (W \to T) \mapsto (W \to T \to S)$ which is faithful. In
  particular, we may view an $S$-scheme $T$ as an
  $S$-groupoid. Thus, a morphism of $S$-schemes $g : U \to V$
  induces a $1$-morphism of $S$-groupoids $\SCH{g} : \SCH{U} \to
  \SCH{V}$. The converse is also true: any $1$-morphism of
  $S$-groupoids $G : \SCH{U} \to \SCH{V}$ is {uniquely} isomorphic
  to a $1$-morphism of the form $\SCH{g}$ for some morphism of
  $S$-schemes $g : U \to V$.
\end{ex}
\begin{defn}
  For an $S$-groupoid $X$, an $X$-\fndefn{scheme} is a pair
  $(T,\sigma_T)$ consisting of an $S$-scheme $T$ together with a
  $1$-morphism of $S$-groupoids $\sigma_T : \SCH{T} \to X$. A morphism
  of $X$-schemes $(f,\alpha_f) : (U,\sigma_U) \to (V,\sigma_V)$ is
  given by a morphism of $S$-schemes $f : {U} \to {V}$ together with a
  $2$-morphism $\alpha_f : \sigma_U \Rightarrow \sigma_V\circ
  \SCH{f}$. The collection of all $X$-schemes forms a $1$-category,
  which we denote as $\SCH{X}$.
\end{defn}
For a $1$-morphism of $S$-groupoids $\Phi : Y \to Z$ there is an
induced functor $\SCH{\Phi} : \SCH{Y} \to \SCH{Z}$. It is readily seen
that for an $S$-groupoid $X$, the 
category $\SCH{X}$ is also an $S$-groupoid. The content
of the $2$-Yoneda Lemma is essentially that the natural $1$-morphism
of $S$-groupoids $\SCH{X} \to X$ is an equivalence. An inverse to this
equivalence is given by picking a clivage for $X$. 

The principal advantage of working with the fibered category 
$\SCH{X}$ is that it admits a \emph{canonical}
clivage. In practice, this means that given an 
$X$-scheme $V$, 
and an $S$-scheme $U$, then for a morphism of $S$-schemes $p : U \to
V$, the way to make $U$ an $X$-scheme is already chosen for us: it is the
composition $\SCH{U} \xrightarrow{\SCH{p}} \SCH{V} \to X$. It is for
this reason that working with $\SCH{X}$ greatly simplifies proofs and
definitions. Calculations, however, are typically easier to perform in
$X$. 

Given a class $P$ 
of morphisms of $S$-schemes and an $S$-groupoid $X$, then a morphism
of $X$-schemes $p : U \to V$ is said to be $P$ if the underlying
morphism of $S$-schemes is $P$.  The following definition is a trivial
generalization of the ideas of M.\ Olsson \cite[App.\ A]{MR2097359},
J. Starr   \cite[\S2]{starr-2006}, and J.\ Wise
\cite[\S2]{2011arXiv1111.4200W}.  
\begin{defn}[$P$-Homogeneity]
  Fix a scheme $S$ and a class $P$ of morphisms of $S$-schemes. A
  $1$-morphism of $S$-groupoids $\Phi : Y\to 
  Z$ is \fndefn{$P$-homogeneous} if the following conditions are
  satisfied.
  \begin{enumerate}
  \item[$(\mathrm{H}^P_1)$] A commutative diagram in the category of
    $Y$-schemes 
    \[ 
    \xymatrix@-0.8pc{T \ar@{^(->}[r]^{i} \ar[d]_p & T'  \ar[d] \\ V
      \ar[r]  & W, }
    \]
    where ${i}$ is a locally nilpotent closed immersion and $p$ is $P$,
    is cocartesian in the category of $Z$-schemes if and only if it is
    cocartesian in the category of $Y$-schemes.
  \item[$(\mathrm{H}^P_2)$] If a diagram of $Y$-schemes $[V
   \xleftarrow{p} T \xrightarrow{{i}} T']$,  where
   ${i}$ is a locally nilpotent closed immersion and $p$ is $P$,
   admits a colimit in the category of $Z$-schemes, then
   there exists a commutative diagram of $Y$-schemes:
   \[ 
    \xymatrix@-0.8pc{T \ar@{^(->}[r]^{i} \ar[d]_p & T'  \ar[d] \\ V
      \ar[r]  &  W.}
    \]
  \end{enumerate}
  An $S$-groupoid $X$ is $P$-homogeneous if its
  structure $1$-morphism is $P$-homogeneous.   
\end{defn}
For homogeneity, we will be interested in the following
classes of morphisms:
\begin{itemize}
  \item[$\HNIL$] -- locally nilpotent closed immersions,
  \item[$\HCL$] -- closed immersions,
  \item[$\HrNIL$] -- morphisms $T \to V$ such that there exists $(T_0
    \to T) \in \HNIL$ with the composition $(T_0 \to T \to V) \in \HNIL$,
  \item[$\HrCL$] -- morphisms $T \to V$ such that there exists $(T_0
    \to T) \in \HNIL$ with the composition $(T_0 \to T \to V) \in \HCL$,
  \item[$\HA$] -- affine morphisms.
\end{itemize}
By \cite[IV.18.12.11]{EGA} universal homeomorphisms are integral, thus
affine. Hence, it is readily deduced that we have a containment 
of classes of morphisms of $S$-schemes:
\[
\xymatrix@-1.5pc{ & \HCL \ar@{}[dr]|-*[@]{\subset}&  &\\ \HNIL
  \ar@{}[ur]|-*[@]{\subset} \ar@{}[dr]|-*[@]{\subset} & &\HrCL
  \ar@{}[r]|-*[@]{\subset}& \HA. \\ &
  \HrNIL \ar@{}[ur]|-*[@]{\subset} &  & }
\]
In \cite[App.\ A]{hallj_dary_artin_crit} it is shown that if $X$ is
limit preserving, in the sense of \cite[\S1]{MR0399094}, and a stack
for the Zariski topology, then $\HrCL$-homogeneity is equivalent to
the condition (S1$'$) of \cite[2.3]{MR0399094}.

J. Wise \cite[Prop.\ 2.1]{2011arXiv1111.4200W} has shown that every 
algebraic stack is $\HA$-homogeneous. In Appendix
\ref{app:schlessinger_pushouts}, we generalize results of D.\ Ferrand
\cite{MR2044495} and obtain techniques to prove that many
``geometric'' moduli problems are $\HA$-homogeneous. We record for
frequent future reference the following 
\begin{lem}\label{lem:homog_pushouts_int}
  Fix a scheme $S$, a class of morphisms $P\subset \HA$, a
  $P$-homogeneous $S$-groupoid $X$, and a diagram of
  $X$-schemes $[V \xleftarrow{p} T \xrightarrow{{i}} T']$, where
  ${i}$ is a locally nilpotent closed immersion and $p$ is
  $P$. Then, there exists a cocartesian diagram in the category of
  $X$-schemes: 
  \[
  \xymatrix@-0.8pc{T \ar@{^(->}[r]^{i} \ar[d]_p & T'  \ar[d]^{p'} \\ V
      \ar@{^(->}[r]^{{i}'}  & V'. }
  \]
  This diagram is also cocartesian in the category of
  $S$-schemes, the morphism ${i}'$ is a locally nilpotent closed
  immersion, $p'$ is affine, and the induced homomorphism of sheaves: 
  \[
  \Orb_{V'} \to {i}'_* \Orb_V \times_{p'_*{i}_*\Orb_T} p'_*\Orb_{T'}
  \]
  is an isomorphism. 
\end{lem}
\begin{proof}
  By Proposition \ref{prop:schlessinger_pushouts_new} (or \cite[Thm.\
  7.1]{MR2044495}) there is a cocartesian diagram in the
  category of $S$-schemes:
  \[
  \xymatrix@-0.8pc{T \ar@{^(->}[r]^{i} \ar[d]_p & T'  \ar[d]^{p'} \\ V
      \ar@{^(->}[r]^{{i}'}  & V'. }
  \]
  The morphism ${i}'$ is a locally nilpotent closed
  immersion, $p'$ is affine, and the induced homomorphism of sheaves
  $\Orb_{V'} \to {i}'_* \Orb_V \times_{p'_*{i}_*\Orb_T}
  p'_*\Orb_{T'}$ is an isomorphism. By Condition  
  $(\mathrm{H}^P_2)$ for $X$, there is thus a commutative diagram of
  $X$-schemes: 
  \[
  \xymatrix@-0.8pc{T \ar@{^(->}[r]^{i} \ar[d]_p & T'  \ar[d] \\ V
      \ar[r]  & W. }
  \]
  Taking the image of this diagram in the category of $S$-schemes, the
  universal property of the colimit $V'$ in the category of
  $S$-schemes produces a unique $S$-morphism $V'\to W$ which makes
  everything commute, giving $V'$ the structure of an
  $X$-scheme. The $S$-morphisms ${i}'$ and $p'$ are promoted to 
  $X$-morphisms, and our original diagram becomes a commutative diagram
  in the category of $X$-schemes. Condition $(\mathrm{H}_1^P)$
  now implies that it is cocartesian in the category of $X$-schemes.  
\end{proof}
The following definition is a convenient computational tool. A
$1$-morphism of $S$-groupoids $\Phi : Y\to Z$ is \fndefn{formally 
  \'etale} if for any $Z$-scheme $T'$ and any locally nilpotent closed immersion of
$Z$-schemes $T \hookrightarrow T'$, then any $Y$-scheme structure on
$T$ which is compatible with its $Z$-scheme structure under $\Phi$,
lifts uniquely to a compatible $Y$-scheme structure on $T'$. That is,
there is always a unique solution to the following lifting problem: 
\[
\xymatrix@-0.8pc{T \ar@{_(->}[d] \ar[r] & Y \ar[d]^{\Phi} \\ T'
  \ar@{-->}[ur]^{\exists\,!}\ar[r] & Z.} 
\]
\begin{lem}\label{lem:homog_prop}
Fix a scheme $S$, a $1$-morphism of $S$-groupoids $\Phi : Y \to Z$,
and a class $P \subset \HA$ of morphisms of $S$-schemes.
\begin{enumerate}
  \item\label{lem:homog_prop:item:comp} If $\Phi$ is
    $P$-homogeneous, then for any other $P$-homogeneous $1$-morphism
    $W \to Y$, the composition $W \to Z$ is $P$-homogeneous.
  \item\label{lem:homog_prop:item:univ} If $Z$ is
    $P$-homogeneous, then a cocartesian diagram of $Y$-schemes:  
    \[ 
    \xymatrix@-0.8pc{T \ar@{^(->}[r]^{i} \ar[d]_p & T'  \ar[d]^{p'} \\ V
      \ar[r]^-{{i}'}  & V', }
    \]
    where ${i}$ is a locally nilpotent closed
    immersion and $p$ is $P$, is also cocartesian in the category of
    $Z$-schemes.  
  \item\label{lem:homog_prop:item:rel_abs} If $Z$ is $P$-homogeneous,
    then the $1$-morphism $\Phi$ is
    $P$-homogeneous if and only if for any $Z$-scheme
    $T$, the $T$-groupoid $Y\times_Z (\SCH{T})$ is $P$-homogeneous.
  \item\label{lem:homog_prop:item:bc} If $Z$ and $\Phi$ are
    $P$-homogeneous, then for any $P$-homogeneous $1$-morphism of
    $S$-groupoids $\Psi : W \to Z$, the $1$-morphism $Y\times_Z W \to
    W$ is $P$-homogeneous.  
  \item\label{lem:homog_prop:item:diag} If $Z$ and $\Phi$ are
    $P$-homogeneous, then the diagonal $1$-morphism $\Delta_\Phi : Y \to
    Y\times_Z Y$ is $P$-homogeneous.
  \item\label{lem:homog_prop:item:fet} If $Z$ is $P$-homogeneous and
    $\Phi$ is  formally \'etale, then $\Phi$ is $P$-homogeneous. 
\end{enumerate}
\end{lem}
\begin{proof}
  For \itemref{lem:homog_prop:item:univ}, by Lemma
  \ref{lem:homog_pushouts_int} the diagram of 
  $Y$-schemes $[V \xleftarrow{p} T\xrightarrow{{i}} T']$ fits into
  a cocartesian diagram of $Z$-schemes:
  \[
  \xymatrix@-0.8pc{T \ar@{^(->}[r]^{{i}} \ar[d]_p & T' \ar[d] \\ V
    \ar[r] & \widetilde{V}.}
  \]
  The universal property defining this square gives a unique map of
  $Z$-schemes $\widetilde{V} \to V'$. Since $V'$ is a $Y$-scheme,
  $\widetilde{V}$ becomes a $Y$-scheme, and the diagram above is promoted 
  to a commutative diagram of $Y$-schemes. We now apply the universal
  property defining $V'$ and obtain a unique morphism of
  $Y$-schemes $V' \to \widetilde{V}$. The morphisms of $Y$-schemes
  $V'\rightleftarrows \widetilde{V}$ are readily seen to be mutually
  inverse. The remainder of the claims are straightforward.
\end{proof}
\section{Extensions}\label{sec:extensions}
The results of this section are well-known to
experts, and similar to those obtained by H. Flenner \cite{MR638811}
and J. Wise \cite[\S2.3]{2011arXiv1111.4200W}.

Fix a scheme $S$ and an $S$-groupoid $X$. An $X$-\fndefn{extension} is  
a square zero closed immersion of $X$-schemes ${i} : T
\hookrightarrow T'$. An obligatory observation
is that the ${i}^{-1}\Orb_{T'}$-module $\ker ({i}^{-1}
\Orb_{T'} \to \Orb_T)$ is canonically a quasicoherent 
$\Orb_T$-module. If the $X$-scheme $T$ is affine, so is 
the scheme $T'$ \cite[{I}.5.1.9]{EGA}. A morphism of
$X$-extensions $({i}_1 : T_1 \hookrightarrow T_1' )  \to 
({i}_2 : T_2 \hookrightarrow T_2')$ is a commutative diagram of
$X$-schemes: 
\[
\xymatrix@R-0.8pc{T_1 \ar@{^(->}[r]^{{i}_1}\ar[d] & \ar[d] T_1' \\ T_2
  \ar@{^(->}[r]^{{i}_2} & T_2'.} 
\]
In a natural way, the collection of $X$-extensions forms a category,
which we denote as $\EXAL_X$. There is a natural functor $\EXAL_X \to
\SCH{X} : ({i} : T\hookrightarrow T') \to T$. 

We denote by $\EXAL_X(T)$ the fiber
of the category $\EXAL_X$ over the $X$-scheme $T$. An
\fndefn{$X$-extension of $T$} is an object of $\EXAL_X(T)$. There is a
natural functor:  
\[
\EXAL_X(T)^\opp \to \QCOH{T} : ({i} : T \hookrightarrow T')
\mapsto \ker({i}^{-1}\Orb_{T'} \to \Orb_T).
\]
We denote by $\EXAL_X(T,{I})$ the fiber category of $\EXAL_X(T)$ over 
the quasicoherent $\Orb_T$-module ${I}$. An \fndefn{$X$-extension of $T$ by
$I$} is an object of $\EXAL_X(T,I)$. A morphism $(T 
\hookrightarrow T'_1) \to (T \hookrightarrow T'_2)$ in $\EXAL_X(T,I)$
induces a commutative diagram of sheaves of rings on the
topological space $|T|$:
\[
\xymatrix@-1pc{0 \ar[r] & I \ar@{=}[d]\ar[r] & \ar[d] \Orb_{T'_2} \ar[r] &
  \Orb_T \ar[r] \ar@{=}[d] &  0\\
0 \ar[r] & I \ar[r] & \Orb_{T'_1} \ar[r] & \Orb_T \ar[r] & 0.}
\]
The Snake Lemma implies that the morphism of $S$-schemes $T_1' \to
T_2'$ is an isomophism, thus the category $\EXAL_X(T,I)$ is a 
groupoid. The following is a triviality that we record here for future
reference. 
\begin{lem}\label{lem:exal_fet}
  Fix a scheme $S$, a formally \'etale $1$-morphism of $S$-groupoids
  $X \to Y$, an $X$-scheme $T$, and a quasicoherent $\Orb_T$-module
  $I$. Then, the natural functor:
  \[
  \EXAL_X(T,I) \to \EXAL_Y(T,I)
  \]
  is an equivalence of categories. 
\end{lem}
Fix a scheme $W$ and a quasicoherent $\Orb_W$-module $J$. Then, the
quasicoherent $\Orb_W$-module $\Orb_W\oplus J$ is readily seen to be a
ring: for an open subset $U \subset |W|$ and $(w,j)$, $(w',j') \in
\Gamma(U,\Orb_W)$ we set 
\[
(w,j)\cdot (w',j') = (ww',wj' + w'j).
\]
Moreover, via the natural map $\Orb_W \to \Orb_W \oplus J : w \mapsto
(w,0)$, we see that the ring $\Orb_W \oplus J$ admits a canonical
structure as an $\Orb_W$-algebra, which we denote as
$\Orb_W\extn{J}$. We now set $W\extn{J}$ to be the $W$-scheme
$\underline{\spec}_W (\Orb_W \extn{J})$. Corresponding to the natural
surjection of $\Orb_W$-algebras $\Orb_W\extn{J} \to \Orb_W$, we obtain
a canonical $W$-extension of $W$ by $J$, which we denote as
$(\trvext{W}{J}: W\hookrightarrow W\extn{J})$, and call the
\fndefn{trivial} $W$-extension of $W$ by $J$. In particular, the
structure morphism $\trvret{W}{J} : W\extn{J} \to W$ is a retraction
of the morphism $\trvext{W}{J} : W \to W\extn{J}$.

For a morphism of $X$-schemes $q : U \to V$, denote by $\Ret_X(U/V)$ the
set of $X$-retractions to the morphism $q : U \to V$. That is,
\[
\Ret_X(U/V) = \{ r\in \Hom_{\SCH{X}}(V,U)\suchthat rq = \ID{U}\}.
\]
\begin{lem}\label{lem:exal_secs}
  Fix a scheme $S$, an $S$-groupoid $X$, an $X$-scheme
  $T$, a quasicoherent $T$-module $I$, and an $X$-extension
  $({i}: T\hookrightarrow T')$ of $T$ by $I$. Then, there is a
  natural bijection:
  \[
  \Hom_{\EXAL_X(T,I)}(({i} : T\hookrightarrow  T') , (\trvext{T}{I} : T
  \hookrightarrow T\extn{I})) \to \Ret_X(T/T').  
  \]
\end{lem}
\begin{proof}
  For a morphism of $X$-extensions $(T \hookrightarrow T') \to
  (T\hookrightarrow T\extn{I})$, the composition $T'
  \to T\extn{I} \xrightarrow{\trvret{T}{I}} T$ defines an
  $X$-retraction to ${i}$. This assignment is bijective.   
\end{proof}
Assuming some homogeneity really gets us something.
\begin{prop}\label{prop:exalmod}
  Fix a scheme $S$, an $S$-groupoid $X$, and an
  $X$-scheme $T$, then the functor $\EXAL_X(T) \to \QCOH{T}^\opp$ is
  a fibration in groupoids. If the $S$-groupoid $X$ is
  $\HNIL$-homogeneous, then $\forall I \in \QCOH{T}$, $\EXAL_X(T,I)$
  is a Picard category.   
\end{prop}
\begin{proof}
  Fix a morphism $\alpha^\opp : J \to I$ in $\QCOH{T}^\opp$. This
  corresponds to a morphism of quasicoherent $\Orb_T$-modules $\alpha
  : I \to J$. Also, fix an $X$-extension  
  $({i} : T \hookrightarrow T_I')$ of $T$ by $I$. On the
  topological space $|T|$ we obtain a commutative diagram of sheaves
  of abelian groups with exact rows:   
  \[
  \xymatrix@-1pc{0 \ar[r] & I \ar[d]_\alpha\ar[r] &
    \ar[d]^{\widetilde{\alpha}} \Orb_{T'_I} \ar[r] & 
    \Orb_T \ar[r] \ar@{=}[d] &  0\\
    0 \ar[r] & J \ar[r] & \Orb_{T'_I} \oplus_I J \ar[r] & \Orb_T \ar[r] & 0,}
  \]
  where 
  \[
  \Orb_{T'_I} \oplus_I J = \coker(I\xrightarrow{i\mapsto (i,-\alpha(i))}
  \Orb_{T'_I} \oplus J). 
  \]
  It is easily verified that the sheaf of abelian groups
  $\Orb_{T'_J} := \Orb_{T'_I}\oplus_I J$ is a sheaf of rings and that the
  homomorphism $\widetilde{\alpha}$ is a ring homomorphism.  
  The subsheaf $J \subset \Orb_{T'_J}$ defines a square zero
  sheaf of ideals and as $J$ is $\Orb_T$-quasicoherent, one
  immediately concludes that the ringed space $(|T|,\Orb_{T'_J})$ is an
  $S$-scheme, $T'_J$, and that we have defined an $S$-extension
  $({i}_\alpha : T \hookrightarrow T'_J)$ of $T$ by $J$. 
  The morphism of $S$-schemes $T'_J \to T'_I$ promotes the
  $S$-extension ${i}_\alpha$ to an 
  $X$-extension of $T$ by $J$. It is immediate that the resulting
  arrow ${i}_\alpha \to {i}$ in $\EXAL_X(T)$ is strongly
  cartesian over the arrow $\alpha^\opp : J \to I$ in $\QCOH{T}^\opp$,
  and we deduce the first claim. 

  For the second claim, the fibration $\EXAL_X(T) \to \QCOH{T}^\opp$
  induces for $M$, $N\in \QCOH{T}$, a functor:
  \[
  \pi_{M,N} : \EXAL_X(T,M\times N) \to \EXAL_X(T,M) \times
  \EXAL_X(T,N).
  \]
  Note that this functor is not unique, but for any other choice of
  such a functor $\pi'_{M,N}$, there is a unique natural
  isomorphism of functors $\pi_{M,N} \Rightarrow \pi'_{M,N}$. This
  renders the Picard category structure on $\EXAL_X(T,I)$ as
  essentially unique (on the level of isomorphism classes of objects,
  the abelian group structure is unique). 
  
  By \cite[\S1.2]{MR0241495} it is sufficient to
  show that the functor $\pi_{M,N}$ is an equivalence, which we show
  using the arguments of \cite[${0}_{\text{IV}}$.18.3]{EGA}. For the 
  essential surjectivity, we fix $X$-extensions $({i}_M : T 
  \hookrightarrow T_M')$ and $({i}_N : T
  \hookrightarrow T_N')$ of $T$ by $M$ and $N$ respectively. Since
  $X$ is $\HNIL$-homogeneous, by Lemma \ref{lem:homog_pushouts_int},
  there is a cocartesian diagram in the category of $X$-schemes:
  \[
  \xymatrix@-0.8pc{T \ar@{^(->}[r]^{{i}_M} \ar@{_(->}[d]_{{i}_N}
    & T_M'  \ar[d]\\ T_N' \ar[r] & T'.} 
  \]
  The resulting closed immersion ${i} : T \hookrightarrow T'$
  defines an $X$-extension of $T$  by $M\times N$. Moreover, it is
  plain to see that $\pi_{M,N}({i}) \cong
  ({i}_M,{i}_N)$. The full faithfulness of the functor
  $\pi_{M,N}$ follows from a similar argument. 
\end{proof}
Denote the set of isomorphism classes of the category
$\EXAL_X(T,I)$ by $\Exal_X(T,I)$. By Proposition \ref{prop:exalmod},
if $X$ is $\HNIL$-homogeneous, there are additive functors: 
\begin{align*}
  \Der_X(T,-) &: \QCOH{T} \to \AB : I \mapsto
  \Aut_{\EXAL_X(T,I)}(\trvext{T}{I})\\
  \Exal_X(T,-) &: \QCOH{T} \to \AB : I \mapsto \Exal_X(T,I). 
\end{align*}
We note that the $0$-object of the abelian group $\Der_X(T,I)$
corresponds to the identity automorphism, and the $0$-object of the
group $\Exal_X(T,I)$ corresponds to the isomorphism class containing
the $X$-extension $(\trvext{T}{I} : T \hookrightarrow
T\extn{I})$. Increasing the homogeneity, more structure is obtained.
\begin{cor}\label{cor:6term}
  Fix a scheme $S$, a $\HrNIL$-homogeneous $S$-groupoid $X$, and an $X$-scheme
  $T$. Then, for each short exact sequence of quasicoherent
  $\Orb_T$-modules: 
  \[
  \xymatrix{0 \ar[r] & K \ar[r] & M \ar[r] & C \ar[r] & 0}
  \]
  there is a natural $6$-term exact sequence of abelian groups:
  \[
  \xymatrix{0 \ar[r] & \Der_{X}(T,K) \ar[r] &
    \Der_{X}(T,M) \ar[r] &
    \Der_{X}(T,C) \ar `[r] `[l] `[dlll]_\partial `[d] [dll] &  &\\
    & \Exal_{X}(T,K) \ar[r] & \Exal_{X}(T,M)
    \ar[r] & \Exal_{X}(T,C).}
  \]
\end{cor}
\begin{proof}
  This is actually a consequence of \cite[Prop.\
  2.3(iv)]{2011arXiv1111.4200W}, where it was shown that the fibered
  category $\EXAL_X(T) \to \QCOH{T}^\opp$ is additive and left-exact,
  in the sense of \cite{MR0241495}. We will not follow this route, but 
  instead utilize arguments similar to
  \cite[${0}_{\text{IV}}$.20.2.2-3]{EGA}.  We will also only  prove
  the exactness of the last three terms, since this is all that is
  necessary in this paper.  

  Given an $X$-extension
  $({i}_M : T \hookrightarrow T'_M)$ of $T$ by $M$, suppose that its
  image, $({i}_C : T\hookrightarrow T'_C)$, in $\Exal_X(T,C)$ is
  $0$. By Lemma \ref{lem:exal_secs}, this is equivalent to the
  existence of an $X$-retraction $r : T'_C \to T$ of the $X$-morphism
  ${i}_C$. Proposition \ref{prop:exalmod} implies that there is
  an induced $X$-morphism $T'_C \to T'_M$. Since the $\Orb_T$-module
  homomorphism $M \to C$ is 
  surjective with kernel $K$, it follows that the $X$-morphism $(T'_C
  \hookrightarrow T'_M)$ defines an $X$-extension of $T'_C$ by 
  $K$. Since $X$ is $\HrNIL$-homogeneous, Lemma
  \ref{lem:homog_pushouts_int} implies that there is a cocartesian
  diagram in the category $X$-schemes:
  \[
  \xymatrix@-0.8pc{T'_C \ar@{^(->}[r] \ar[d]_r & T'_M \ar[d] \\ T
    \ar@{^(->}[r]^{{i}} & T'.}
  \]
  Certainly, $({i} : T\hookrightarrow T')$ defines an $X$-extension
  of $T$ by $K$ and the image of the $X$-extension ${i}$ in
  $\Exal_X(T,M)$ is readily seen to be ${i}_M$. 
\end{proof}
Strengthening our homogeneity assumption again, we see more.
\begin{cor}\label{cor:et_loc_exal}
  Fix a scheme $S$, an $\HA$-homogeneous $S$-groupoid $X$, and an
  $X$-scheme $T$. For any \emph{affine} and \'etale morphism $p : U
  \to T$, and any quasicoherent  $\Orb_U$-module $J$, there is an
  equivalence of Picard categories:  
  \[
     \EXAL_X(U,J) \to \EXAL_X(T,p_*J).
   \]
\end{cor}
\begin{proof}
  First, we observe that given any \'etale morphism $q : V \to T$ and
  an $X$-extension $T\hookrightarrow T'$ of $T$ by $K$, then by
  \cite[{IV}.18.1.2]{EGA}, there exists a unique $X$-extension
  $V\hookrightarrow V'$ of $V$ by $q^*K$ together with an \'etale
  morphism $V' \to T'$ such that $V'\times_{T'} T \cong V$ and the
  second projection is the map $V \to T$. This describes a functor
  $q^* : \EXAL_X(T,K) \to \EXAL_X(V,q^*K)$. Applying this with
  $K=p_*J$, we obtain a functor $\EXAL_X(T,p_*J) \to
  \EXAL_X(U,p^*p_*J)$. Applying Proposition \ref{prop:exalmod} to the
  $\Orb_U$-module homomorphism $p^*p_*J \to J$, there is an induced
  functor $\EXAL_X(U,p^*p_*J) \to \EXAL_X(U,J)$. Composing these two 
    functors produces a functor $\EXAL_X(T,p_*J) \to \EXAL_X(U,J)$. 

  Also, since the morphism $p : U \to 
  T$ is affine, $\HA$-homogeneity implies that
  there is a functor $p_*: \EXAL_X(U,J) \to \EXAL_X(T,p_*J)$. Indeed,
  given an $X$-extension $(U\hookrightarrow U')$ of $U$ by $J$, the
  $\HA$-homogeneity of $X$, combined with Lemma
  \ref{lem:homog_pushouts_int}, gives a cocartesian diagram of $X$-schemes:  
  \[
  \xymatrix@-0.8pc{U \ar@{^(->}[r] \ar[d]_p & U' \ar[d] \\ T \ar@{^(->}[r]
    & T'.}
  \]
  It is readily verified that the $X$-morphism $(T \hookrightarrow
  T')$ defines an $X$-extension of $T$ by $p_*J$. The functors
  $\EXAL_X(T,p_*J) \rightleftarrows  \EXAL_X(U,J)$ are 
  clearly quasi-inverse. 
\end{proof}
\section{Limit preservation}\label{sec:limit_preservation}
In this section we prove that the functors defined in
\S\ref{sec:extensions}, $M\mapsto \Der_X(T,M)$ and $M\mapsto
\Exal_X(T,M)$, frequently 
preserve direct limits. We also relativize the notion of
{limit preserving} $S$-groupoid \cite[\S1]{MR0399094}. 
\begin{defn}
  Fix a scheme $S$. A $1$-morphism of $S$-groupoids $\Phi : Y \to Z$
  is \fndefn{limit preserving} if given an inverse system of
  quasicompact and quasiseparated $Z$-schemes with affine transition
  maps $\{W_j\}_{j\in J}$, as well as a $Y$-scheme $V$, such that as a
  $Z$-scheme it is an inverse limit of $\{W_j\}_{j\in J}$, then there
  exists $j_0\in J$ and an essentially unique $Y$-scheme structure on
  $W_{j_0}$ (i.e.~for any two choices and all $j\gg j_0$ the two
  induced $Y$-scheme structures on $W_j$ are isomorphic) such that the
  induced diagram of $Y$-schemes $\{W_j\}_{j\geq j_0}$ has limit $V$.
  An $S$-groupoid $X$ is \fndefn{limit preserving} if its structure
  morphism to $\SCH{S}$ is so. Similarly, an $X$-scheme $T$ is
  \fndefn{limit preserving} if its structure $1$-morphism $\SCH{T} \to
  X$ is so.
\end{defn} 
Analogous to Lemma \ref{lem:homog_prop}, we have the following easily
verified lemma. 
\begin{lem}\label{lem:lim_prop}
  Fix a scheme $S$ and a $1$-morphism of $S$-groupoids $\Phi : Y \to Z$.
  \begin{enumerate}
  \item\label{lem:lim_prop:item:artin} If $Z$ is a
    Zariski stack, then it is limit preserving if and only if for any
    inverse system of affine $S$-schemes $\{\spec A_j\}_{j\in J}$
    with limit $\spec A$, the natural functor:
    \[
    \varinjlim_j Z(\spec A_j) \to Z(\spec A)
    \]
    is an equivalence.
  \item\label{lem:lim_prop:item:rep} If $Z$ is an
    algebraic stack, then it is limit preserving if and only if it is
    locally of finite presentation over $S$. 
  \item\label{lem:lim_prop:item:comp} If $\Phi$ is
    limit preserving, then for any other limit preserving $1$-morphism
    $W \to Y$, the composition $W \to Z$ is  limit preserving.
  \item\label{lem:lim_prop:item:rel_abs} The $1$-morphism $\Phi$ is
    limit preserving if and only if for any $Z$-scheme $T$, the
    $T$-groupoid $Y\times_Z \SCH{T}$ is limit preserving.
  \item\label{lem:lim_prop:item:bc} If $\Phi$ is limit preserving,
    then for any $1$-morphism of  $S$-groupoids $W \to Z$, the
    $1$-morphism $Y\times_Z W \to W$ is   limit preserving.  
  \item\label{lem:lim_prop:item:diag} If $\Phi$ is
    limit preserving, then the diagonal $1$-morphism $\Delta_\Phi : Y \to
    Y\times_Z Y$ is limit preserving.
  \end{enumerate}
\end{lem}
\begin{proof}
  The only non-obvious point is \itemref{lem:lim_prop:item:rep}, which
  follows from \cite[4.15--18]{MR1771927}. 
\end{proof}
\begin{ex}
  Fix a scheme $S$ and a limit preserving $S$-groupoid $X$. Then, an
  $X$-scheme is limit preserving if and only if it is locally of 
  finite presentation over $S$.  
\end{ex}
We now have the main result of this section.
\begin{prop}\label{prop:exalfp}
  Fix a scheme $S$, a $\HNIL$-homogeneous $S$-groupoid $X$, and a
  quasicompact, quasiseparated, limit preserving
  $X$-scheme $T$. 
  \begin{enumerate}
  \item \label{prop:exalfp:item:der} The functor $M\mapsto \Der_X(T,M)$
    preserves direct limits.  
  \item \label{prop:exalfp:item:ex} If, in addition, $X$ is limit
    preserving, then the functor $M\mapsto \Exal_X(T,M)$ 
    preserves direct limits. 
  \end{enumerate}
\end{prop}
\begin{proof}
Throughout, we fix a directed system of quasicoherent $\Orb_T$-modules
$\{M_j\}_{j \in J}$ with direct limit $M$. Certainly, in the category of
$X$-schemes the natural map $T\extn{M} \to \varprojlim_j T\extn{M_j}$
is an isomorphism. For \itemref{prop:exalfp:item:der}, by Lemma
\ref{lem:exal_secs} we have:
\begin{align*}
  \Der_X(T,M) &= \Ret_{X}(T/T\extn{M}) = \varinjlim_j
  \Ret_X(T/T\extn{M_j}) = \varinjlim_j \Der_X(T,M_j).
\end{align*}
For \itemref{prop:exalfp:item:ex}, we first show that the map
$\varinjlim_j \Exal_X(T,M_j) \to \Exal_X(T,M)$ is injective. Lemma
\ref{lem:exal_secs} shows that an $X$-extension 
$(T\hookrightarrow T'')$ of $T$ by a quasicoherent $\Orb_T$-module
$N$ represents $0$ in $\Exal_X(T,N)$ if and only if $\Ret_X(T/T'')
\neq \emptyset$. So, given a compatible collection of
$X$-extensions $(T\hookrightarrow T_j')$ of $T$ by $M_i$, with
limit $(T\hookrightarrow T')$, then since $\Ret_X(T/T') = \varinjlim_j
\Ret_X(T/T_j')$, we deduce that the map $\varinjlim_j \Exal_X(T,M_j) \to
\Exal_X(T,M)$ is injective.

We now show that the natural map $\varinjlim_j \Exal_X(T,M_j) \to
\Exal_X(T,M)$ is surjective. First, we prove the result in the case
where $X=S$ and $S$ and $T$ are affine. Since $T$ is affine and of
finite presentation over $S$, there exists an integer $n$ and a closed
immersion $k : T\hookrightarrow \Aff^n_S$. By
\cite[$0_{\mathrm{IV}}$.20.2.3]{EGA}, there is a functorial surjection
for every $\Orb_T$-module $K$:
$\Hom_{\Orb_T}(k^*\Omega_{\Aff^n_S/S},K) \twoheadrightarrow
\Exal_S(T,K)$. Since the $\Orb_T$-module $k^*\Omega_{\Aff^n_S/S}$ is
finite free, it follows that the functor $K\mapsto
\Hom_{\Orb_T}(k^*\Omega_{\Aff^n_S/S},K)$ preserves direct
limits. Direct limits are exact so we have a surjection
$\varinjlim_j\Exal_S(T,M_j) \twoheadrightarrow \Exal_S(T,M)$. 

If $S$ and $T$ are no longer assumed to be affine, a straightforward Zariski
descent argument, combined with the affine case already considered,
shows that we also have a bijection $\varinjlim_j \Exal_S(T,M_j) \to
\Exal_S(T,M)$. Now for the general case: given $(T \hookrightarrow T')
\in \Exal_X(T,M)$, by the above considerations there exists a $j_0$
and an $S$-extension of $T$ by $M_{j_0}$, $(T\hookrightarrow
T'_{j_0})$, such that its pushforward along $M_{j_0} \to M$ is
isomorphic to $(T\hookrightarrow T')$ as an $S$-extension. If $j\geq
j_0$, denote the pushforward of $(T\hookrightarrow T'_{j_0})$ along
the morphism $M_{j_0} \to M_j$ by $(T \hookrightarrow T'_j)$. There is
a natural morphism of $S$-schemes $T'_j \to T'_{j_0}$ and the resulting
inverse system $\{T'_j\}_{j\geq j_0}$ has limit $T'$. Since $X$ is a
limit preserving $S$-groupoid, there exists $j_1 \geq j_0$ and an
$X$-scheme structure on $T'_{j_1}$ such that the resulting inverse
system of $X$-schemes $\{T'_{j}\}_{j\geq j_1}$ has limit
$T'$. The result follows. 
\end{proof}
\section{Formal smoothness and formal versality}\label{sec:fs_fv}
In this section we prove the main result of the paper. 
\begin{defn}
  Fix a scheme $S$, an $S$-groupoid $X$, and an $X$-scheme $T$. Consider
  the following lifting problem: given a square zero closed immersion
  of $X$-schemes $Z_0 \hookrightarrow Z$ fitting into a
  commutative diagram of $X$-schemes:
  \[
  \xymatrix{ Z_0 \ar@{_(->}[d] \ar[r]^g  & T \ar[d]  \\ Z
    \ar[r] \ar@{-->}[ur]   & X.}
  \]
  We say that the $X$-scheme $T$ is
  \begin{mydescription}
  \item[\fndefn{formally smooth}] if the lifting problem above 
    can always be solved \'etale locally on $Z$;
  \item[\fndefn{formally versal at $t\in |T|$}] if the lifting problem
    can be solved whenever $Z$ is local artinian, with
    closed point $z$, such that $g(z) = t$, $\kappa(z) \cong
    \kappa(t)$, and there is an isomorphism 
    of $\Orb_T$-modules $\kappa(t) \cong g_*\ker (\Orb_{Z}\to
    \Orb_{Z_0})$.
  \end{mydescription}
\end{defn}
We certainly have the following implication:
\[
{\text{formally smooth} \Rightarrow \text{formally versal at all $t\in
    |T|$}.} 
\]
In general, there is no reverse implication. We will see,
however, that this subtlety vanishes once the $S$-groupoid is
$\HA$-homogeneous.
\begin{ex}
Fix an $S$-groupoid $X$ and an $X$-scheme $T$ such that the
$1$-morphism $T \to X$ is representable by algebraic spaces which are
locally of finite presentation. Then, the $X$-scheme $T$ is formally
smooth if and only if the $1$-morphism $T\to X$ is representable by
smooth morphisms of algebraic spaces. 
\end{ex}
There is a tight connection between formal smoothness
(resp.\ formal 
versality) and $X$-extensions in the \emph{affine} setting. The next
result has arguments similar to those of \cite[Satz 3.2]{MR638811},
but the definitions are slightly different. 
\begin{lem}\label{lem:smooth}
  Fix a scheme $S$, an $S$-groupoid $X$, and an \emph{affine}
  $X$-scheme $T$. 
  \begin{enumerate}
  \item\label{lem:item:smooth:fs} If $X$ is
    $\HA$-homogeneous 
    and the abelian group 
    $\Exal_{X}(T,M)$ is trivial for all quasicoherent $\Orb_T$-modules
    $M$, then the $X$-scheme $T$ is 
    formally smooth.
  \item\label{lem:item:smooth:fv_suff} If $X$ is
    $\HrCL$-homogeneous and at a 
    {closed} point $t\in |T|$, $\Exal_X(T,\kappa(t)) = 0$,
    then the $X$-scheme $T$ is formally versal at $t$. 
  \item\label{lem:item:smooth:fv_nec} If $X$ is $\HCL$-homogeneous
    and $T$ is noetherian and formally versal at a {closed} point $t
    \in |T|$, then $\Exal_{X}(T,\kappa(t))=0$.    
 \end{enumerate}
\end{lem}
\begin{proof}
For \itemref{lem:item:smooth:fs}, fix a
square zero closed immersion $Z_0 \hookrightarrow Z$ 
(defined by a quasicoherent $\Orb_{Z_0}$-module $I$) of
$X$-schemes, fitting into a commutative diagram: 
\[
\xymatrix@-0.8pc{Z_0 \ar[r]_g \ar@{_(->}[d]& T \ar[d] \\ Z \ar[r] & X.}
\]
We need to construct an $X$-morphism $Z \to T$ \'etale
locally on $Z$. Thus we easily 
reduce to the case where $Z_0$, $Z$, and $T$ are  
affine. Lemma \ref{lem:homog_pushouts_int} now gives a cocartesian
diagram of $X$-schemes: 
\[
\xymatrix@-0.8pc{Z_0 \ar[r]_g \ar@{_(->}[d]& T \ar@{^(->}[d] \\ Z
  \ar[r] & T',} 
\]
where the $X$-morphism $T \to T'$ defines an $X$-extension of $T$
by $g_*I$. 
By hypothesis, $\Exal_X(T,g_*I)=0$, and
Lemma  \ref{lem:exal_secs} produces an $X$-retraction $T' \to
T$. The composition $Z \to T' \to T$ gives the required
lifting. The claim \itemref{lem:item:smooth:fv_suff} follows from an
identical argument just given for \itemref{lem:item:smooth:fs}.  

For \itemref{lem:item:smooth:fv_nec}, given an $X$-extension
$T \hookrightarrow T'$ of $T$ by 
$\kappa(t)$, write $T=\spec R$, $T'=\spec R'$, $\mathfrak{m} = t\in
|T|$, and $I = \ker (R'\to R) \cong 
R/\mathfrak{m}$. Let the ideal $\mathfrak{m}' 
\ideal R'$ denote the 
(unique) maximal ideal induced by $\mathfrak{m}$. For $n\geq 0$
define $R_n = R/\mathfrak{m}^{n+1}$, $R_n' =
R'/\mathfrak{m}'^{n+1}$, and $I_n = \ker (R_n' \to R_n)$. The
following diagram commutes:  
\[
\xymatrix@-0.8pc{\spec R_n \ar@{^(->}[d] \ar[r] & T \ar@{=}[r]
  \ar@{^(->}[d] & T \ar[d] \\ \spec R_n' \ar[r] & T' \ar[r] & X.} 
\]
Formal versality at $t\in |T|$ gives for each $n\geq 0$ an
$X$-morphism $\spec R_n' \to T$ completing the diagram.
For each $n\geq 0$ there is also 
a cocartesian diagram of $X$-schemes (Lemma~\ref{lem:homog_pushouts_int}): 
\[
\xymatrix@-0.8pc{\spec R_n \ar@{^(->}[d] \ar[r] & T
  \ar@{^(->}[d]\\ \spec R_n' \ar[r] & \tilde{T}_n.} 
\]
Thus, an $X$-morphism $\spec R_n'
\to T$ induces a unique $X$-retraction $\tilde{T}_n \to T$ to the
$X$-extension $T \hookrightarrow \tilde{T}_n$. Moreover, there is
a unique morphism of
$X$-extensions $\alpha : (T\hookrightarrow \tilde{T}_n) \to (T
\hookrightarrow T')$. Since
the $R$-module $I$ is of length $1$, it follows that for $n \gg 0$ the
surjective map $I \to I_n$ is an isomorphism. Thus, the morphism
$\alpha$ is an  isomorphism for $n\gg 0$ and the $X$-extension 
$T\hookrightarrow T'$ admits an $X$-retraction. By Lemma
\ref{lem:exal_secs}, $\Exal_X(T,\kappa(t)) = 0$. 
\end{proof}
\begin{rem}
  With some additional work and some finiteness assumptions, it is
  possible to prove the converse to Lemma
  \ref{lem:smooth}\eqref{lem:item:smooth:fs}. 
\end{rem}
Fix an affine scheme $T$ and an additive functor $F : \QCOH{T} \to
\AB$. The functor $F$ is \fndefn{finitely generated} if there exists a
quasicoherent $\Orb_T$-module $I$ and an object $\eta \in F(I)$ such 
that for all $M\in \QCOH{T}$, the induced morphism of abelian groups
$\Hom_{\Orb_T}(I,M) \to F(M) : f\mapsto f_*\eta$ is surjective. The
notion of finite generation of a functor is due to M. Auslander
\cite{MR0212070}. 

The functor $F$ is
\fndefn{half-exact} if for any short exact sequence in $\QCOH{T}$, $0
\to M' \to M \to M'' 
\to 0$, the sequence $F(M') \to F(M) \to F(M'')$ is exact. 

If, in addition, the scheme $T$ is noetherian, and $F$ is
half-exact, sending coherent $\Orb_T$-modules to coherent
$\Orb_T$-modules, then A. Ogus and G. Bergman have shown
\cite[Thm.\ 2.1]{MR0302633} that if for all closed points $t
\in |T|$ we have $F(\kappa({t})) = 0$, then $F$ is the zero
functor. If $F$ is finitely generated, then this result can be
refined. Indeed, it is shown in \cite[Cor.~6.7]{hallj_coho_bc} that
if $F(\kappa(t)) = 0$, then there exists an affine open
subscheme $p : U \hookrightarrow T$ such that the composition $F \circ
p_*(-) : \QCOH{U} \to \AB$ is identically zero. We now use this
to prove the main technical result of the paper.  
\begin{thm}\label{thm:opennessversal}
  Fix a locally noetherian scheme $S$, an $\HA$-homogeneous and limit
  preserving $S$-groupoid $X$, and an affine $X$-scheme 
  $T$, locally of finite type over $S$. If the functor $M
  \mapsto \Exal_X(T,M)$ is 
  finitely generated and $T$ is formally versal at a closed point
  $t\in |T|$, then it is formally smooth in an open neighbourhood of
  $t$.   
\end{thm}
\begin{proof}
  By Lemma \ref{lem:smooth}\eqref{lem:item:smooth:fv_nec},
  $\Exal_X(T,\kappa(t)) = 0$. By Corollary \ref{cor:6term} the functor
  $M\mapsto \Exal_X(T,M)$ is half-exact, and by Proposition
  \ref{prop:exalfp} it commutes with direct limits. As
  $\Exal_X(T,-)$ is finitely generated, \cite[Cor.~6.7]{hallj_coho_bc}
  now applies. Thus, there exists an affine open 
  neighbourhood $p :U 
  \hookrightarrow T$ of $t$ such that the functor
  $\Exal_X(T,p_*(-)) : \QCOH{U} \to \AB$ is the zero functor. By
  Corollary \ref{cor:et_loc_exal}, $\Exal_X(U,-)$ is also the zero
  functor. By Lemma \ref{lem:smooth}\eqref{lem:item:smooth:fs}, we
  conclude that $U$ is a formally smooth $X$-scheme.
\end{proof}
We will defer the proof of the following Corollary until
\S\ref{sec:crit_alg} as we currently lack the necessary computational
tools (e.g.\ the relationship between $\Exal$ and
$\Def$).
\begin{cor}\label{cor:repcrit1}
  Fix an excellent scheme $S$. An $S$-groupoid $X$ is an
  algebraic $S$-stack, locally of finite presentation over $S$, if 
  and only if the following conditions are satisfied.
  \begin{enumerate}
  \item $X$ is a stack over the site $(\SCH{S})_\Et$.
  \item $X$ is limit preserving.
  \item $X$ is $\HA$-homogeneous.
  \item The diagonal $\Delta_{X/S} : X \to X\times_S X$ is
    representable by algebraic spaces.
  \item For any local noetherian ring $(B,\mathfrak{m})$, such that
    the ring $B$ is $\mathfrak{m}$-adically complete, with an
    $S$-scheme structure $\spec B \to S$ such that the induced
    morphism $\spec (B/\mathfrak{m}) \to S$ is locally of finite type,
    then the natural functor:
    \[
    \FIB{X}{\spec B} \to \varprojlim_n \FIB{X}{\spec (B/\mathfrak{m}^{n})}
    \]
   has dense image.
  \item For any affine $X$-scheme $T$, locally of finite
    type over $S$, the functor $M \mapsto \Exal_X(T,M)$ is finitely
    generated.   
  \end{enumerate}
\end{cor}
\section{Coherent functors}\label{sec:coherent}
Fix a ring $A$. An additive functor $F:\MOD{A} \to \AB$ is
\fndefn{coherent}, if 
there exists an $A$-module homomorphism $f :I \to J$ and an element
$\eta \in F(I)$, inducing an exact sequence for any $A$-module $M$:  
\[
\xymatrix{\Hom_A(J,M) \ar[r] & \Hom_A(I,M) \ar[r] & F(M) \ar[r] & 0. }
\]
We refer to the data $(f : I \to J,\eta)$ as a \fndefn{presentation}
for $F$. For a comprehensive account of coherent functors, we refer the
interested reader to \cite{MR0212070}. Some stronger results that are 
available in the noetherian situation are developed in
\cite{MR1656482}.  Here we record some simple consequences of
\cite[Prop.\ 2.1]{MR0212070}.  
\begin{lem}\label{lem:coherent_ab}
  Fix a ring $A$. For each $i=1$, $\dots$, $5$, let $H_i : \MOD{A} \to
  \AB$ be an additive functor fitting into an exact sequence:
  \[
  \xymatrix{H_1 \ar[r] & H_2 \ar[r] & H_3 \ar[r] & H_4 \ar[r] & H_5. }
  \]
  \begin{enumerate}
  \item If $H_2$, $H_4$ are finitely generated, and $H_5$ is coherent,
    then $H_3$ is finitely generated.
  \item If $H_1$, $H_2$ are finitely generated, and $H_4$, $H_5$
    are coherent, then $H_3$ is coherent
  \end{enumerate}
\end{lem}
We now have two fundamental examples. 
\begin{ex}\label{ex:noetherianexalder}
  Fix a scheme $S$ and a locally noetherian algebraic $S$-stack
  $X$. Let $T$ be an affine and noetherian $X$-scheme, which is locally of
  finite type. Then, the functors $M\mapsto \Der_X(T,{M})$
  and $M\mapsto 
  \Exal_X(T,{M})$ are coherent. Indeed, by
  \cite[Thm.\ 1.1]{MR2206635}, 
  there is a bounded above complex of $\Orb_T$-modules $L_{T/X}$, with
  coherent cohomology sheaves, as well as functorial isomorphisms
  $\Der_X(T,{M}) \cong  
  \Ext^0_{\Orb_T}(L_{T/X},M)$ and $\Exal_X(T,{M}) \cong
  \Ext^1_{\Orb_T}(L_{T/X},M)$  for 
  all quasicoherent $\Orb_T$-modules $M$. The claim now follows from
  \cite[Ex.\ 3.13]{hallj_coho_bc}.
\end{ex}
The next example is \cite[Thm.\ C]{hallj_coho_bc}.  
\begin{ex}\label{ex:coh_func_thm}
  Fix an affine scheme $S$ and a morphism of algebraic stacks $f : X
  \to S$ which is separated and locally of finite presentation. Let If
  $\shv{M}$, $\shv{N} \in \QCOH{X}$, with $\shv{N}$ of finite
  presentation, flat over $S$, with support proper over $S$, then for
  all $i\geq 0$ the functor:
  \[
  \Ext^i_{\Orb_X}(\shv{M},\shv{N}\tensor_{\Orb_X} f^*(-)) : \QCOH{S}
  \to \AB
  \]
  is coherent. 
\end{ex}
\section{Automorphisms, deformations, obstructions, and
  composition}\label{sec:def_obs_comp} 
A hypothesis in Theorem \ref{thm:opennessversal} is that the 
functor $M \mapsto \Exal_X(T,M)$ is finitely generated. We have found
the direct verification of this hypothesis to be difficult. In this
section, we provide some exact sequences to remedy this situation. We
also take the opportunity to formalize and relativize obstruction theories. 

Fix a scheme $S$ and a $1$-morphism of $S$-groupoids $\Phi : Y \to
Z$. Define the category $\DEF_\Phi$ to have objects the triples
$(T,J,\eta)$, where $T$ is a $Y$-scheme, $J$ is a
quasicoherent $\Orb_T$-module, and 
$\eta$ is a $Y$-scheme structure on the trivial $Z$-extension of $T$ by
$J$. A morphism $(T,J,\eta) \to (V,K,\xi)$ consists of a $Y$-scheme
morphism $f : T\to V$, a morphism of quasicoherent $\Orb_T$-modules
$f^*K \to J$ such that the induced morphism of trivial $Z$-extensions
$(T\hookrightarrow T\extn{J}) \to (V\hookrightarrow V\extn{K})$ is a
morphism of $Y$-extensions. Graphically, it is the category of
completions of the following diagram: 
\[
\xymatrix{\ar@{_(->}[d]T \ar[r] & Y \ar[d]^{\Phi} \\ 
T\extn{J} \ar@{-->}[ur]^\eta \ar[r] & Z.}
\]
There is a natural functor $\DEF_\Phi \to \SCH{Y} : (T,J,\eta)
\mapsto T$. We denote the fiber of this functor over the
$Y$-scheme $T$ by $\DEF_\Phi(T)$. There is also a functor
$\DEF_\Phi(T)^\opp \to \QCOH{T} : (J,\eta) \mapsto J$. We denote the fiber
of this functor over a quasicoherent $\Orb_T$-module $J$ as
$\DEF_{\Phi}(T,J)$. This category is naturally pointed by the trivial
$Y$-extension of $T$ by $J$. Also, if the $1$-morphism $\Phi$ is fibered in setoids,
then the category $\DEF_\Phi(T,J)$ is discrete. Another observation
is that if $\Phi_T$ denotes the $T$-groupoid $\Phi\times_Z
T$, then the natural functor 
\[
\DEF_\Phi(T,J) \to \DEF_{\Phi_T}(T,J)
\]
is an equivalence. We record for future reference the following
trivial observations.
\begin{lem}\label{lem:def_fet}
  Fix a scheme $S$, $1$-morphisms of $S$-groupoids $X
  \xrightarrow{\Psi} Y \xrightarrow{\Phi} Z$, an $X$-scheme $T$, and a
  quasicoherent $\Orb_T$-module $I$. If 
  the $1$-morphism $\Psi: X \to Y$ is formally \'etale, then the
  natural functor: 
  \[
  \DEF_{\Phi\circ\Psi}(T,I) \to \DEF_{\Phi}(T,I).
  \]
  is an equivalence of categories. 
\end{lem}
\begin{lem}\label{lem:def_Phomog}
  Fix a scheme $S$, a class of morphisms $P \subset \HA$,
  a $1$-morphism of $P$-homogeneous $S$-groupoids $\Phi : Y \to Z$, a
  morphism of $Y$-schemes $p : U \to V$ where $p\in P$, and $K \in
  \QCOH{U}$. Then, the natural functor:
  \[
  \DEF_{\Phi}(V,p_*K) \to \DEF_\Phi(U,K),
  \]
  is an equivalence of categories. 
\end{lem}
The proof of the next result is similar to Proposition
\ref{prop:exalmod}, thus is omitted. 
\begin{prop}\label{prop:def_pic}
  Fix a scheme $S$, a $1$-morphism of $\HNIL$-homogeneous $S$-groupoids
  $\Phi : Y \to Z$, a $Y$-scheme $T$, and a quasicoherent
  $\Orb_T$-module $J$. Then the 
  category $\DEF_\Phi(T,J)$ admits a natural structure as a Picard
  category. 
\end{prop}
Denote the set of isomorphism classes of the Picard category
$\DEF_\Phi(T,J)$ by $\Def_\Phi(T,J)$. Thus, by Proposition
\ref{prop:def_pic}, we obtain functors:
\begin{align*}
  \Def_\Phi(T,-) &: \QCOH{T} \to \AB : J \mapsto
  \Def_\Phi(T,J)\\
  \Aut_\Phi(T,-) &: \QCOH{T} \to \AB : J \mapsto
  \Aut_{\DEF_\Phi(T,J)}(\trvext{T}{J}).  
\end{align*}
The proof of the next result is similar to Corollary
\ref{cor:6term}. We will not be using this result, however, so we omit
the proof.
\begin{cor}\label{cor:6term_def}
  Fix a scheme $S$, a $1$-morphism of $\HrNIL$-homogeneous $S$-groupoids
  $\Phi: Y \to Z$, and a $Y$-scheme
  $T$. Then, for each short exact sequence in $\QCOH{T}$:
  \[
  \xymatrix{0 \ar[r] & K \ar[r] & M \ar[r] & C \ar[r] & 0}
  \]
  there is a natural exact sequence of abelian groups:
  \[
  \xymatrix{0 \ar[r] & \Aut_{\Phi}(T,K) \ar[r] &
    \Aut_{\Phi}(T,M) \ar[r] &
    \Aut_{\Phi}(T,C) \ar `[r] `[l] `[dlll] `[d] [dll] &  &\\
    & \Def_{\Phi}(T,K) \ar[r] & \Def_{\Phi}(T,M)
    \ar[r] & \Def_{\Phi}(T,C).}
  \]
\end{cor}
We now have a simple result whose proof we leave to the
conscientious reader.
\begin{prop}\label{prop:derdefseq}
  Fix a scheme $S$, a $1$-morphism of $\HNIL$-homogeneous $S$-groupoids
  $\Phi : Y \to Z$, a $Y$-scheme $T$, and a quasicoherent
  $\Orb_T$-module $J$. Then, 
  there is a natural exact sequence of abelian groups:
    \[
    \xymatrix{0 \ar[r] & \Aut_{\Phi}(T,J) \ar[r] &
    \Der_{Y}(T,J) \ar[r] &
    \Der_{Z}(T,J) \ar `[r] `[l] `[dlll] `[d] [dll] &  &\\
    & \Def_{\Phi}(T,J) \ar[r] & \Exal_{Y}(T,J)
    \ar[r] & \Exal_{Z}(T,J).}
  \]
\end{prop}
We now introduce multi-step relative
obstruction theories. For single-step obstruction theories, this
definition is similar to \cite[2.6]{MR0399094} and
\cite[A.10]{MR2097359}.  
\begin{defn}
  Fix a scheme $S$, a $1$-morphism of $\HNIL$-homogeneous $S$-groupoids
  $\Phi : Y \to Z$, and an integer $n\geq 1$. For a $Y$-scheme $T$, an 
  \fndefn{$n$-step relative obstruction theory for $\Phi$ at $T$}
  is a sequence of additive functors (the 
  obstruction spaces): 
  \[
  \mathrm{O}^i(T,-) : \QCOH{T} \to \AB : J \mapsto
  \mathrm{O}^i(T,J) \quad i=1,\dots,n
  \]
  as well as natural transformations of functors (the obstruction maps):
  \begin{align*}
    \mathrm{o}^1(T,-) &: \Exal_Z(T,-) \Rightarrow
    \mathrm{O}^1(T,-)  \\ 
    \mathrm{o}^i(T,-) &: \ker \mathrm{o}^{i-1}(T,-)
    \Rightarrow \mathrm{O}^i(T,-) \quad\mbox{for $i=2$, $\dots$, $n$},
  \end{align*}  
  such that the natural transformation of functors:
  \[
  \Exal_Y(T,-) \Rightarrow \Exal_Z(T,-)
  \]
  has image $\ker \mathrm{o}^n(T,-)$. For an \emph{affine} $Y$-scheme
  $T$, an $n$-step relative obstruction theory at $T$ is  
  \fndefn{coherent} if the functors
  $\{\mathrm{O}^i(T,-)\}_{i=1}^n$ are all  
  coherent.
\end{defn}
We feel that it is important to point out that simply taking the
cokernel of the last morphism in the exact sequence of Proposition
\ref{prop:derdefseq} produces a $1$-step relative
obstruction theory, which we
denote as $(\obs_{\Phi}, {\Obs}_\Phi)$, and call the \fndefn{minimal}
relative obstruction theory. This obstruction theory
generalizes to the relative setting the minimal obstruction theory described in 
\cite{MR638811}. In practice, the minimal obstruction theory is 
a difficult object to explicitly describe. Now, combining Lemmata
\ref{lem:def_fet} and \ref{lem:exal_fet} we obtain 
\begin{lem}\label{lem:obs_fet}
  Fix a scheme $S$, $1$-morphisms of $\HNIL$-homogeneous $S$-groupoids
  $X \xrightarrow{\Psi} Y \xrightarrow{\Phi} Z$, an $X$-scheme $T$, and a
  quasicoherent $\Orb_T$-module $I$. If $\Psi$ is formally \'etale,
  then any $n$-step relative obstruction theory for $\Phi$ at $T$ lifts
  to an $n$-step relative obstruction theory for $\Phi\circ \Psi$ with
  the same obstruction spaces.
\end{lem}
What follows is an immediate consequence of Proposition  
\ref{prop:derdefseq} and Lemma \ref{lem:coherent_ab}.
\begin{cor}\label{cor:derdefseq}
  Fix a scheme $S$, a $1$-morphism of $\HNIL$-homogeneous
  $S$-groupoids $\Phi : Y \to Z$, an affine
  $Y$-scheme $T$, and an integer $n\geq 1$. Suppose there exists a
  coherent $n$-step relative obstruction theory at $T$.
  \begin{enumerate}
  \item If the functor $M\mapsto \Exal_Z(T,{M})$ is finitely
    generated, then the minimal obstruction theory 
    $(\obs_\Phi,{\Obs}_\Phi)$ is
    coherent at $T$. 
  \item If the functors $M \mapsto \Def_\Phi(T,M)$,
    $\Exal_Z(T,M)$  
    are finitely generated, then the functor $M \mapsto
    \Exal_Y(T,M)$ 
    is finitely generated.
  \end{enumerate}
\end{cor}
This next result summarizes, in the conventions of this paper, some
well-known results from the literature. As can be seen, the relative
situation is clarifying. The result that follows also shows the
stability of the conditions of Theorem \ref{mainthms:repcrit2} under
composition, in the sense of J. Starr \cite{starr-2006}. 
\begin{prop}\label{prop:defexactseq}
  Fix a scheme $S$ and $1$-morphisms of $\HNIL$-homogeneous $S$-groupoids
  $X \xrightarrow{\Psi} Y \xrightarrow{\Phi} Z$, an $X$-scheme $T$, and a
  quasicoherent $\Orb_T$-module $I$.  
  \begin{enumerate}
  \item \label{prop:defexactseq:item:6term} There is a natural $9$-term
    exact sequence of abelian groups:
    \[
    \xymatrix{0 \ar[r] & \Aut_{\Psi}(T,I) \ar[r] &
      \Aut_{\Phi\circ\Psi}(T,I) \ar[r] &
      \Aut_{\Phi}(T,I) \ar `[r] `[l] `[dlll] `[d] [dll] &  \\
      & \Def_{\Psi}(T,I) \ar[r] & \Def_{\Phi\circ \Psi}(T,I)
      \ar[r] & \Def_{\Phi}(T,I)  \ar `[r] `[l] `[dlll] `[d] [dll] &  \\
      & {\Obs}_\Psi(T,I)\ar[r] & {\Obs}_{\Phi\circ
        \Psi} (T,I) \ar[r] & {\Obs}_\Phi(T,I) \ar[r] & 0.
    } 
    \]
  \item\label{prop:defexactseq:item:autdefsobs} There are natural
    isomorphisms of abelian groups:
    \[
    \Aut_\Psi(T,I) \to \Def_{\Delta_\Psi}(T,I) \quad \mbox{and} \quad 
    \Def_\Psi(T,I) \to \Obs_{\Delta_\Psi}(T,I).
    \]
    In particular, we may realize the functor $I\mapsto \Def_\Psi(T,I)$ as a
    $1$-step relative obstruction theory for the $1$-morphism $\Delta_\Psi$.
  \item \label{prop:defexactseq:item:obs_bc} Fix a $\HNIL$-homogeneous
    $1$-morphism of $S$-groupoids $W\to Y$, 
    an $X_W$-scheme $U$, and a quasicoherent $\Orb_U$-module $J$. Then
    there is a natural injection 
    \[
    \Obs_{\Psi_W}(U,J) \subset
    \Obs_\Psi(U,J).
    \]
    In particular, we may realize the functor $J\mapsto \Obs_\Psi(U,J)$
    as a $1$-step relative obstruction theory for the $1$-morphism
    $\Psi_W : X_W \to W$. 
  \end{enumerate}
\end{prop}
\begin{proof}
  For \itemref{prop:defexactseq:item:6term}, we first apply the Snake
  Lemma to the commutative diagram: 
  \[
  \xymatrix{&\Exal_X(T,I) \ar[r] \ar[d] & \ar[d]\Exal_Y(T,I) \ar[r] & 
    \ar[d] {\Obs}_{\Psi}(T,I) \ar[r] & 0\\ 
    0 \ar[r] & \Exal_Z(T,I) \ar[r] & \Exal_Z(T,I) \ar[r] &
    0.&} 
  \]
  Combining this with Proposition \ref{prop:derdefseq} produces an
  exact sequence: 
  \[
  \xymatrix{\Def_\Phi(T,I) \ar[r] & {\Obs}_{\Psi}(T,I)
    \ar[r] & {\Obs}_{\Phi\circ \Psi}(T,I) \ar[r] &
    {\Obs}_{\Phi}(T,I) \ar[r] & 0.}
  \]
  A direct argument, as in \cite[A.15]{MR2097359}, produces the
  first $7$ terms of the exact sequence. Splicing these together
  gives the result. 

  The claim \itemref{prop:defexactseq:item:autdefsobs} follows from 
  \itemref{prop:defexactseq:item:6term} upon taking $\Psi :=
  \Delta_\Psi$, $\Phi$ the first projection $X\times_Y X\to X$, and
  noting that $\Aut_{\ID{X}} = \Def_{\ID{X}} = 0$. 
  
  For \itemref{prop:defexactseq:item:obs_bc}, we note that
  \itemref{prop:defexactseq:item:6term} provides a natural 
  homomorphism of abelian groups $\Obs_{\Psi_W}(U,J) \to
  \Obs_{X_W/Y}(U,J) \to  \Obs_{\Psi}(U,J)$. To see that this
  composition of maps is 
  injective, suppose that we have a $W$-extension $(U\hookrightarrow
  U')$ of $U$ by $J$. If it lifts, as a $Y$-extension, to an
  $X$-extension, then the universal property of the $2$-fiber
  product implies that it lifts to an $X_W$-extension. This proves
  the claim. 
\end{proof}
\section{Proof of Theorem \ref{mainthms:repcrit2}}\label{sec:crit_alg}
In this section we prove Theorem \ref{mainthms:repcrit2}. Before we do
this, however, we prove Corollary \ref{cor:repcrit1}.
\begin{proof}[Proof of Corollary \ref{cor:repcrit1}]
  Fix a morphism $x : \spec \Bbbk \to S$, where $\Bbbk$ is a
  field. Denote by $\mathscr{A}_S(x)$ the category whose objects
  are pairs $(A,a)$, where $A$ is a local artinian ring with residue
  field $\Bbbk$, and $a : \spec A \to S$ is a morphism of schemes, such
  that the composition $\spec \red{A} \to \spec A \to S$ agrees with
  $x$. Morphisms $(A,a) \to (B,b)$ in $\mathscr{A}_S(x)$ are ring
  homomorphisms $A \to B$ preserving the data. For  $\xi \in
  \FIB{X}{x}$, there is an induced category fibered in groupoids
  $X_\xi : \mathscr{C}_\xi \to 
  \mathscr{A}_S(x)^\opp$. The $\HA$-homogeneity of the 
  $S$-groupoid $X$ implies the homogeneity (in the sense of
  \cite[Exp.\ VI, Defn.\ 2.5]{SGA7}) of the cofibered category
  $X_\xi^\opp : \mathscr{C}_\xi^\opp \to \mathscr{A}_S(x)$. 

  If the morphism $x$ is locally of finite type, then by (6) and
  \cite[Lem.~6.6]{hallj_coho_bc} the $\Bbbk$-vector space 
  $\Exal_X(\xi,\Bbbk)$ is finite dimensional. By Example
  \ref{ex:noetherianexalder} and \loccit the
  $\Bbbk$-vector space $\Der_S(x,\Bbbk)$ is finite dimensional,
  and thus by Proposition \ref{prop:derdefseq}, the $\Bbbk$-vector
  space $\Def_{X/S}(\xi,\Bbbk)$ is finite dimensional. By definition,
  $\Def_{X/S}(\xi,\Bbbk)$ is the set of isomorphism classes of 
  the category $\FIB{X_\xi}{\xi[\epsilon]}$. 

  Thus, by (5), \cite[Thm.\ 1.5]{MR1935511} applies, and so
  for any such $\xi$, there is a pointed and affine $X$-scheme
  $(Q_\xi,q)$, locally of finite type over $S$, such that the $X$-scheme
  $\spec \kappa(q)$ is isomorphic to $\xi$,
  and $Q_\xi$ is formally versal at $q$. We now apply Theorem
  \ref{thm:opennessversal} to conclude that we may (by passing to an
  open subscheme) assume that $Q_\xi$ is a formally smooth $X$-scheme
  containing $q$. Condition (4) implies that the $X$-scheme $Q_\xi$ is
  representable by  smooth morphisms. 

  Define $K$ to be the set of all morphisms
  $x : \spec \Bbbk \to S$ which are locally of finite type, and where
  $\Bbbk$ is a field. Set
  $Q := \amalg_{\kappa\in K,\xi\in \FIB{X}{\kappa}} Q_\xi$. Then, we
  have seen that the $X$-scheme $Q$ is representable by smooth
  morphisms, and it remains to show that it is representable by
  surjective morphisms. Since the stack $X$ is limit preserving, it is
  sufficient to test this claim with affine $X$-schemes 
  $V$ which are of locally of finite type over $S$. The
  morphism of algebraic $S$-spaces $Q\times_X V \to V$ is smooth, and
  by construction its image contains all the points
  $v\in |V|$ such that the morphism $\spec \kappa(v) \to S$ is locally
  of finite type. Since, $V$ is of locally of finite type over $S$, it
  follows that $Q\times_X V \to V$ is surjective. 
\end{proof}
Bootstrapping, we can use Corollary  
\ref{cor:repcrit1} to obtain Theorem \ref{mainthms:repcrit2}.
\begin{proof}[Proof of Theorem \ref{mainthms:repcrit2}]
  Note that conditions (1) and (2), combined with Lemma
  \ref{lem:lim_prop}\itemref{lem:lim_prop:item:artin}, imply that the
  $S$-groupoid $X$ is limit preserving. 
  
  Suppose that the diagonal morphism $\Delta_{X/S} : X \to 
  X\times_S X$ is representable. Conditions (5) and (6),
  together with Corollary \ref{cor:derdefseq}, imply that for any
  affine $X$-scheme $V$ which is locally of finite type over $S$, the
  functor $M \mapsto \Exal_X(V,M)$ is finitely generated. Thus,
  Corollary \ref{cor:repcrit1} implies that $X$ is an
  algebraic stack which is locally of finite presentation over $S$.  

  Next, will show that if the second diagonal morphism
  $\Delta_{\Delta_{X/S}} : X \to X \times_{X\times_S X} X$ is
  representable, then the $1$-morphism $\Delta_{X/S} : X \to X\times_S
  X$ is representable by algebraic spaces. By Lemmata 
  \ref{lem:homog_prop}\itemref{lem:homog_prop:item:diag} and  
  \ref{lem:lim_prop}\itemref{lem:lim_prop:item:diag}, the diagonal
  $1$-morphism $\Delta_{X/S} : X \to X\times_S X$ is $\HA$-homogeneous and
  limit preserving. By Lemma
  \ref{lem:homog_prop}(\ref{lem:homog_prop:item:bc}\&\ref{lem:homog_prop:item:comp}), 
  we see that the $S$-groupoid $X\times_S X$ is
  $\HA$-homogeneous. Thus, by Lemmata
  \ref{lem:homog_prop}\itemref{lem:homog_prop:item:bc} and 
  \ref{lem:lim_prop}\itemref{lem:lim_prop:item:bc}, for any $X\times_S   
  X$-scheme $T$, the $T$-groupoid $I_{X,T} := X\times_{X\times_S X}
  (\SCH{T})$ is limit preserving and
  $\HA$-homogeneous. Representability of $I_{X,T}$ is local on $T$ for
  the Zariski topology, thus we may assume that $T$ is an affine
  scheme. By Lemma
  \ref{lem:lim_prop}(\ref{lem:lim_prop:item:bc}\&\ref{lem:lim_prop:item:comp}), 
  the $S$-groupoid $X\times_S X$ is limit preserving, thus any affine
  $X\times_S X$-scheme  $X\times_S X$-scheme $T$ factors through an affine
  $X\times_S X$-scheme $T_0$ that is locally of finite type over
  $S$. Thus, we may assume henceforth that $T$ is locally of finite
  type over $S$, and is consequently excellent. 

  Let $V$ be an affine $I_{X,T}$-scheme that is locally of finite type
  over $T$ (thus locally of finite type over $S$). Then, given
  $I\in \QCOH{V}$, we have natural isomorphisms:
  \[
  \Def_{I_{X,T}/T}(V,I) \cong \Def_{(I_{X,T})_V/V}(V,I) \cong
  \Def_{I_{X,V}/V}(V,I) \cong \Def_{\Delta_{X/S}}(V,I).
  \]
  By Proposition
  \ref{prop:defexactseq}\itemref{prop:defexactseq:item:autdefsobs}, we
  thus have $\Def_{I_{X,T}/T}(V,I) \cong \Aut_{X/S}(V,I)$ and so the functor
  $M\mapsto \Def_{I_{X,T}/T}(V,M)$ is coherent. By Proposition
  \ref{prop:defexactseq}(\ref{prop:defexactseq:item:autdefsobs}\&\ref{prop:defexactseq:item:obs_bc})
  we also have 
  \[
  \Obs_{I_{X,T}/T}(V,I) \subset \Obs_{\Delta_{X/S}}(V,I) \cong
  \Def_{X/S}(V,I). 
  \]
  Hence, the functor $M\mapsto \Def_{X/S}(V,M)$ defines a $1$-step,
  coherent relative obstruction theory for the $1$-morphism $I_{X,T}
  \to T$ at $V$. The $T$-groupoid $I_{X,T}$ has representable
  diagonal, thus satisfies the
  conditions of the previous analysis, hence is an algebraic
  stack, locally of finite presentation over $T$. The diagonal
  $1$-morphism $\Delta_{I_{X,T}/T}$ is a monomorphism, thus $I_{X/T}$
  is an algebraic space. 

  It remains to show that the hypotheses of the Theorem guarantee that
  the second diagonal morphism $\Delta_{\Delta_{X/S}}$ is
  representable. Fix an $X$-scheme $T$, which by the analysis above we
  may assume is locally of finite type over $S$ and excellent, then it
  remains to show that the $T$-groupoid $R_{X,T}:=
  X\times_{(X\times_{X\times_S X} X)} (\SCH{T})$ is representable by
  algebraic spaces. By the previous analysis, we deduce immediately
  that $R_{X,T}$ is a limit preserving and $\HA$-homogeneous
  $T$-groupoid. Also, the third diagonal $1$-morphism of $S$-groupoids
  $\Delta_{\Delta_{\Delta_{X/S}}}$ is an isomorphism, thus is
  representable. So the diagonal $1$-morphism of the $T$-groupoid
  $R_{X,T}$ is an isomorphism. For an affine $R_{X,T}$-scheme $V$
  which is locally of finite type over $S$, and a quasicoherent
  $\Orb_V$-module $I$ we have just shown that $\Def_{R_{X,T}/T}(V,I) =
  0$. By Proposition
  \ref{prop:defexactseq}(\ref{prop:defexactseq:item:autdefsobs}\&\ref{prop:defexactseq:item:obs_bc})
  we see that
  \[
  \Obs_{R_{X,T}/T}(V,I) \subset
  \Obs_{\Delta_{\Delta_{X/S}}}(V,I) \cong \Def_{\Delta_{X/S}}(V,I)
  \cong \Aut_{X/S}(V,I).  
  \]
  Hence, the functor $M\mapsto \Aut_{X/S}(V,M)$ defines a $1$-step
  coherent relative obstruction theory for the $T$-groupoid
  $R_{X,T}$ at $V$. Applying the first
  analysis to this $T$-groupoid proves the result. 
\end{proof}
\section{Application I: the stack of quasicoherent
  sheaves}\label{sec:example_qcoh} 
Fix a scheme $S$. For an algebraic $S$-stack 
$Y$ and a property $P$ of
quasicoherent $\Orb_Y$-modules, denote by 
 $\QCOH[P]{Y}$ the   
full subcategory of $\QCOH{Y}$ consisting of objects which are $P$. We
will be interested in the following properties $P$ of quasicoherent
$\Orb_Y$-modules and their combinations: 
\begin{enumerate}
\item[\textbf{fp}] -- finitely presented,
\item[\textbf{fl}] -- $Y$-flat,
\item[\textbf{flb}] -- $S$-flat,
\item[\textbf{prb}] -- $S$-proper support.
\end{enumerate}
Fix a morphism of algebraic stacks $f : X\to S$. For any $S$-scheme
$T$, consider a property $P$ of     
quasicoherent $\Orb_{X_T}$-modules. Define 
$\QCOHSTK[P]{X}{S}$ to be the category with objects a pair $(T, \shv{M})$,
where $T$ is an $S$-scheme and $\shv{M}\in \QCOH[P]{X_T}$. A morphism
$(a,\alpha) : (V,\shv{N}) \to (T,\shv{M})$ in 
the category $\QCOHSTK[P]{X}{S}$ consists of an $S$-scheme morphism $a :
V\to T$ together with an $\Orb_{X_V}$-{isomorphism} $\alpha :
a^*_{X_T}\shv{M} \to 
\shv{N}$. Set $\COHSTK{X}{S} = \QCOHSTK[\bflat,\fp,\bcompact]{X}{S}$. In
this section we will prove 
\begin{thm}\label{thm:flshvs_alg}
  Fix a scheme $S$ and a morphism of algebraic stacks $f : X \to S$,
  which is separated and locally of finite presentation. Then,
  $\COHSTK{X}{S}$ is an algebraic stack, 
  locally of finite presentation over $S$, with affine diagonal
  over $S$.  
\end{thm}
A proof of Theorem \ref{thm:flshvs_alg}, without the statement about
the diagonal, appeared in \cite[Thm.\ 2.1]{MR2233719}, though was
light on details. In particular, no explicit obstruction theory was
given and, as we will see, the obstruction theory is subtle when $f$
is not flat (and is not a standard fact). There was also a minor error
in the statement---that the morphism $f$ is separated is essential
\cite{MR2369042}. The statement about the diagonal of $\COHSTK{X}{S}$
was addressed by M.~Roth and J.~Starr
\cite[\S2]{2009arXiv0908.0096R}--- their approach, however, is
completely different, and relies on \cite[Prop.~2.3]{MR2233719}. In
the setting of analytic spaces, the properties of the diagonal were
addressed by H.~Flenner \cite[Cor.~3.2]{MR641823}.

Just as in \opcit[, Prop.\ 2.7], an immediate consequence of
Theorems \ref{thm:flshvs_alg} and \cite[Thm.\ D]{hallj_coho_bc} is the
existence of Quot spaces. Recall that for a quasicoherent
$\Orb_X$-module $\shv{F}$, the presheaf $\Quotshf_{X/S}(\shv{F}) :
(\SCH{S})^\opp \to \SETS$ is defined as follows:
\[
\Quotshf_{X/S}(\shv{F})[T\xrightarrow{\tau} S] = \{ \tau_X^*\shv{F}
\twoheadrightarrow \shv{Q} \suchthat \shv{Q} \in
\QCOH[\bflat,\fp,\bcompact]{X_T}\}/\cong.
\]
\begin{cor}\label{cor:quot}
  Fix a scheme $S$ and an algebraic $S$-stack $X$ that is separated
  and locally of finite presentation over $S$. Let $\shv{F}\in
  \QCOH{X}$, then $\Quotshf_{X/S}(\shv{F})$ is an algebraic
  space which is separated over $S$. If, in addition, $\shv{F}$ is of
  finite presentation, then $\Quotshf_{X/S}(\shv{F})$ is locally of
  finite presentation over $S$.   
\end{cor}
When $\shv{F}$ is of finite
presentation, Corollary \ref{cor:quot} was proved by M. Olsson and J. Starr \cite[Thm.\
1.1]{MR2007396} and M. Olsson \cite[Thm.\ 1.5]{MR2183251}. When
$\shv{F}$ is quasicoherent and $X \to S$ is locally projective,
Corollary \ref{cor:quot} was recently addressed by G. Di Brino
\cite[Thm.~0.0.1]{2012arXiv1212.4544D} using different methods. 

To prove Theorem \ref{thm:flshvs_alg}, we
use Theorem \ref{mainthms:repcrit2}. Note that there are inclusions:  
\[
\QCOHSTK[\bflat,\fp,\bcompact]{X}{S} \subset
\QCOHSTK[\bflat,\fp]{X}{S} \subset \QCOHSTK[\bflat]{X}{S}.
\]
The first inclusion is trivially
formally \'etale. By Lemma \ref{lem:nilp_fp}\itemref{item:nilp_fpmod}
the second inclusion is also formally \'etale. Thus, by Lemmata 
\ref{lem:homog_prop}(\ref{lem:homog_prop:item:comp}\&\ref{lem:homog_prop:item:fet})
if $\QCOHSTK[\bflat]{X}{S}$ is $\HA$-homogeneous, the same will be
true of $\COHSTK{X}{S}$. Also, by Lemmata
\ref{lem:def_fet} 
and \ref{lem:obs_fet}, it is sufficient to determine the
automorphisms, deformations, and obstructions for
$\QCOHSTK[\bflat]{X}{S}$. 

Throughout, we fix a clivage for $\QCOHSTK{X}{S}$. This gives an
equivalence of $S$-groupoids $\QCOHSTK{X}{S} \to
\SCH{\QCOHSTK{X}{S}}$, which we will use without further comment. 
\begin{lem}\label{lem:flshvs_homog}
  Fix a scheme $S$ and a morphism of algebraic stacks $f : X \to
  S$. Then, the $S$-groupoid $\QCOHSTK[\bflat]{X}{S}$ is
  $\HA$-homogeneous. 
\end{lem}
\begin{proof}
First we check $(\mathrm{H}_1^{\HA})$. Fix a commutative
diagram of $\QCOHSTK[\bflat]{X}{S}$-schemes: 
\begin{equation}
\xymatrix@-0.5pc{(T_0,\shv{M}_0) \ar@{^(->}[r]^{({i},\phi)} \ar[d]_{(p,\pi)}
  & \ar[d]^{(p',\pi')} (T_1,\shv{M}_1) \\ 
  (T_2,\shv{M}_2) \ar@{->}[r]^-{({i}',\phi')} &
  (T_3,\shv{M}_3),}\label{eq:qcoh_h1} 
\end{equation}
where $p$ is affine and ${i}$ is a locally nilpotent closed
immersion. Set $(g,\gamma) = ({i}',\phi')\circ (p,\pi):
(T_0,\shv{M}_0) \to  (T_3,\shv{M}_3)$. Lemma 
\ref{lem:homog_prop}\itemref{lem:homog_prop:item:univ} implies
that if the diagram \eqref{eq:qcoh_h1} is cocartesian in the category of
$\QCOHSTK[\bflat]{X}{S}$-schemes, then it remains cocartesian in the
category of $S$-schemes. Conversely, suppose that the diagram
\eqref{eq:qcoh_h1} is cocartesian in the category of $S$-schemes. By Lemma
\ref{lem:homog_pushouts_int} (applied to $X=S$), ${i}'$ is a
locally nilpotent closed immersion and $p'$ is affine. Let
$(W,\shv{N})$ be a $\QCOHSTK[\bflat]{X}{S}$-scheme, and for $k\neq 3$
fix $\QCOHSTK[\bflat]{X}{S}$-scheme maps  $(y_k,\psi_k) : 
(T_k,\shv{M}_k) 
\to (W,\shv{N})$. Since the diagram \eqref{eq:qcoh_h1} is cocartesian
in the category of $S$-schemes, there exists a unique $S$-morphism $y_3 :
T_3 \to W$ that is compatible with this data. By adjunction, we
obtain  unique maps of $\Orb_{X_{W}}$-modules: 
\[
\shv{N} \to (y_{1})_*\shv{M}_1 \times_{(y_{0})_*\shv{M}_0} (y_{2})_*\shv{M}_2
\cong \{(y_{3})_*p'_*\shv{M}_1\} \times_{\{(y_{3})_*g_*\shv{M}_0\}}
\{(y_{3})_*{i}'_*\shv{M}_2\}.
\]
The functor $(y_{3})_*$ is left-exact, so there is a
functorial isomorphism $\Orb_{X_W}$-modules: 
\[
\{(y_{3})_*p'_*\shv{M}_1\} \times_{\{(y_{3})_*g_*\shv{M}_0\}}
\{(y_{3})_*{i}'_*\shv{M}_2\} \cong (y_{3})_*\{p'_*\shv{M}_1
\times_{g_*\shv{M}_0} {i}'_*\shv{M}_2\}. 
\]
The commutativity of the diagram \eqref{eq:qcoh_h1} posits a uniquely 
induced morphism:
\[
\delta : \shv{M}_3 \to p'_*\shv{M}_1 \times_{g_*\shv{M}_0} {i}'_*\shv{M}_2 \cong
p'_*p'^*\shv{M}_3 \times_{g_*g^*\shv{M}_3} {i}'_*{i}'^*\shv{M}_3.
\]
Thus, it is sufficient to prove that the map
$\delta$ is an isomorphism, which is local for the
smooth topology. So, we immediately reduce to the affine case, where
$S=\spec A$, $X=\spec D$, and  $f : X \to S$ is given by a ring
homomorphism $A  \to D$. For each $i$ we have $T_i = \spec
A_i$ and we set $D_i = D\tensor_A A_3$. Also, $\shv{M}_3\cong
\widetilde{M}_3$, where $M_3$ is a $D_3$-module which is
$A_3$-flat. Now, we have an exact sequence  of $A_3$-modules:   
\[
\xymatrix{0 \ar[r] & A_3 \ar[r] & A_1 \times A_2 \ar[r] & A_0 \ar[r] & 0.}
\]
Applying the exact functor $M_3\tensor_{A_3} - $ to this sequence
produces an exact sequence: 
\[
\xymatrix{0 \ar[r] & M_3 \ar[r] & (M_3\tensor_{A_3} A_1) \times
  (M_3\tensor_{A_3} A_2) \ar[r] & M_3\tensor_{A_3} A_0 \ar[r] & 0.} 
\]
Since $M_3\tensor_{A_3} A_i \cong M_3 \tensor_{D_3} D_i$, we obtain
the required isomorphism $\delta$: 
\[
M_3 \cong (M_3\tensor_{A_3} A_1) \times_{(M_3\tensor_{A_3} A_0)}
(M_3\tensor_{A_3} A_2) \cong (M_3\tensor_{D_3} D_1)
\times_{(M_3\tensor_{D_3} D_0)} 
(M_3\tensor_{D_3} D_2).
\]

Next we check condition $(\mathrm{H}_2^\HA)$. Fix a diagram of
$\QCOHSTK[\bflat]{X}{S}$-schemes: 
\[
[ (T_1,\shv{M}_1) \xleftarrow{({i},\phi)}
(T_0,\shv{M}_0) \xrightarrow{(p,\pi)} (T_2,\shv{M}_2)],
\]
where ${i}$ is a locally nilpotent closed immersion and $p$ is
affine. Given a cocartesian square of $S$-schemes: 
\begin{equation}
\xymatrix@-0.8pc{T_0 \ar@{^(->}[r]^{i} \ar[d]_p & T_1 \ar[d]^{p'} \\ T_2
  \ar@{^(->}[r]^{{i}'} & T_3, }\label{eq:qcoh_inc}
\end{equation}
write $g = {i}'p$ and set 
\[
\shv{M}_3 = \ker \left( (p'_{X_{T_3}})_*\shv{M}_1 \times
({i}'_{X_{T_3}})_*\shv{M}_2 \xrightarrow{d} g_*\shv{M}_0\right) \in
\QCOH{X_{T_3}}, 
\]
where $d$ is the map $(m_1,m_2) \mapsto (g_*\phi)(m_1) -
(g_*\pi)(m_2)$. It remains to show that $\shv{M}_3$ is
$T_3$-flat, that the induced morphisms of
quasicoherent $\Orb_{X_2}$-modules $\phi' : {i}'^*_{X_{T_3}}\shv{M}_3 \to
\shv{M}_2$ and $\pi' : p'^*_{X_{T_3}}\shv{M}_3 \to \shv{M}_1$ are
isomorphisms, and that the following diagram commutes:
\[
\xymatrix@R-1.5pc{& \ar[dl]{i}^*_{X_{T_1}}p'^*_{X_{T_3}}\shv{M}_3
  \ar[rr]^{{i}^*\pi'} & &  
  {i}_{X_{T_1}}^*\shv{M}_1 \ar[dr]^\phi&\\
g^*\shv{M}_3 & & & & \shv{M}_0\\
& \ar[ul] p_{X_{T_2}}^*{i}'^*_{X_{T_3}}\shv{M}_3 \ar[rr]^{p^*_{X_{T_2}}\phi'} & &
p_{X_{T_2}}^*\shv{M}_2 \ar[ur]_\pi. &}  
\]
Indeed, this shows that the pairs $({i}',\phi')$ and $(p',\pi')$
define $\QCOHSTK[\bflat]{X}{S}$-morphisms, and that the resulting
completion of the diagram \eqref{eq:qcoh_inc} commutes.

Now, these claims may all be verified locally for the smooth topology. Thus,
we reduce to the affine situation as before, with the modification
that for $k\neq 3$ we have $\shv{M}_k\cong \widetilde{M}_k$, where
$M_k$ is a $D_k$-module which is flat over $A_k$, and $\shv{M}_3 \cong
\widetilde{M}_3 \cong \widetilde{M}_1\times_{\widetilde{M}_0}
\widetilde{M}_2$. The result now follows from \cite[Thm.\ 2.2]{MR2044495}.
\end{proof}
We now determine the automorphisms, deformations, and
obstructions. Let $(T,\shv{M})$ be a $\QCOHSTK[\bflat]{X}{S}$-scheme,
and fix a quasicoherent $\Orb_T$-module $I$. For an $S$-extension ${i} :  
T\hookrightarrow T'$ of $T$ by $I$, we have a $2$-cartesian diagram:  
\[
\xymatrix{X_T \ar[d]_{f_T} \ar@{^(->}[r]^{{j}} &
  X_{T'}    \ar[d]^{f_{T'}} \ar[r] & X \ar[d]^f\\ T \ar@/_.9pc/[rr]_\tau
  \ar@{^(->}[r]^{{i}} & T' 
  \ar[r]^{\tau'} & S.}
\]
Set $J = {j}^*\ker(\Orb_{X_{T'}} \to {j}_*\Orb_{X_T})$. Fix a
$\QCOHSTK{X}{S}$-extension $({i},\phi) : (T,\shv{M}) \to
(T',\shv{M}')$, then we obtain a commutative diagram:  
\[
\xymatrix@-0.5pc{\shv{M}\tensor_{\Orb_{X_T}} f_T^*I
  \ar@/_1pc/[dr]\ar@{->>}[r] &  
  \shv{M}\tensor_{\Orb_{X_T}} J \ar@{->>}[d] \\ & {j}^*\ker(\shv{M'}\to
  {j}_*\shv{M}).} 
\]
By the local criterion for flatness, $\shv{M}'$ is $T'$-flat if and
only if the diagonal map is an isomorphism. Thus, if a
$\QCOHSTK[\bflat]{X}{S}$-extension $({i},\phi) : (T,\shv{M}) \to
(T',\shv{M}')$ exists, the top map must be an isomorphism. This is
how we will describe our first obstruction. 
\begin{ex}\label{ex:non_flat_non_split}
  This obstruction can be
  non-trivial when $f$ is not flat and ${i}$ is not split. Indeed,
  let $S = \spec \C[x,y]$ and take $0=(x,y)\in |S|$ to be the
  origin. Set $X = \mathrm{Bl}_{0}S=\underline{\mathrm{Proj}}_S
  \Orb_S[U,V]/(xV-yU)$, $f : X \to S$  
  the induced map, and let $E = f^{-1}(0)$ be the exceptional
  divisor. Now take $\shv{M} = \Orb_E$ and consider the $S$-extension $T=\spec
  \kappa(0) \hookrightarrow T'=\spec \C[x,y]/(x^2,y)$. A
  straightforward calculation shows that $\shv{M}\tensor_{\Orb_{X_T}}
  J$ is the skyscraper sheaf supported at the point of $E$
  corresponding to the $y=0$ line in $S$. Also, $f_T^*I = \Orb_{X_T}$
  and so $\shv{M} \tensor_{\Orb_{X_T}} f_T^*I \cong \Orb_E$. The
  resulting map $\shv{M} \tensor_{\Orb_{X_T}} f_T^*I \to \shv{M}
  \tensor_{\Orb_{X_T}} J$ is not injective. 
\end{ex}
Observe that there is a short exact sequence of $\Orb_{T'}$-modules:
\begin{equation*}
\xymatrix{0 \ar[r] & {i}_*I \ar[r] & \Orb_{T'} \ar[r] &
  {i}_*\Orb_T \ar[r] & 0.}\label{eq:nonsplit}
\end{equation*}
By Theorem \ref{thm:tor} we obtain an exact sequence of
quasicoherent $\Orb_{X_{T'}}$-modules:
\[
\xymatrix{\STor^{S,\tau',f}_1({i}_*\Orb_T,\Orb_X) 
  \ar[r] &  f_{T'}^*{i}_*I \ar[r] &
  {j}_*J \ar[r] & 0.} 
\]
Since we have a functorial isomorphism $f_{T'}^*{i}_*I \cong
{j}_*f_T^*I$, by Lemma \ref{lem:aff_tor} we obtain
a natural exact sequence of quasicoherent $\Orb_{X_T}$-modules: 
\[
\xymatrix{\STor^{S,\tau,f}_1(\Orb_T,\Orb_X) 
  \ar[r] &  f_{T}^*I \ar[r] &
  J \ar[r] & 0.} 
\]
Applying the functor $\shv{M}\tensor_{\Orb_{X_T}}-$ to this sequence 
produces another exact sequence:  
\[
\xymatrix@-.4pc{ \shv{M}   {\tensor}_{\Orb_{X_T}}
  \STor^{S,\tau,f}_1(\Orb_T,\Orb_X)  
  \ar[rrr]^-{\mathrm{o}^1((T,\shv{M}),I)({i})}
  & &  & \shv{M} {\tensor}_{\Orb_{X_T}} f_T^*I \ar[r] &
  \shv{M}{\tensor}_{\Orb_{X_T}} J \ar[r] & 0.}
\]
Thus, we have defined a natural class
\[
\mathrm{o}^1((T,\shv{M}),I)({i}) \in
\Hom_{\Orb_{X_T}}( \shv{M} \tensor_{\Orb_{X_T}} \STor^{S,\tau,f}_1(\Orb_T,\Orb_X) 
, \shv{M} \tensor_{\Orb_{X_T}} f_T^*I),
\]
whose vanishing is necessary and sufficient for the map
$\shv{M}\tensor_{\Orb_{X_T}} f_T^*I \to \shv{M}\tensor_{\Orb_{X_T}} J$ to
be an isomorphism. By functoriality of the class
$\mathrm{o}^1((T,\shv{M}),I)({i})$, we
obtain a natural transformation of functors:  
\begin{align*}
  \mathrm{o}^1((T,\shv{M}),-) : \Exal_S(T,-)  \Rightarrow
  \Hom_{\Orb_{X_T}}( \shv{M} 
  {\tensor}_{\Orb_{X_T}} \STor^{S,\tau,f}_1(\Orb_T,\Orb_X) , \shv{M}
  {\tensor}_{\Orb_{X_T}} f_T^*(-)).
\end{align*}
So, suppose that the $S$-extension ${i} : T\hookrightarrow T'$ now
has the property that the map $\shv{M}\tensor_{\Orb_{X_T}} f_T^*I \to \shv{M}\tensor_{\Orb_{X_T}}
J$ is an isomorphism. Let $\gamma_{\shv{M},I}$ denote the inverse to
this map, then \cite[IV.3.1.12]{MR0491680} gives a naturally defined
obstruction:  
\[
\mathrm{o}^2((T,\shv{M}),I)({i}) \in
\Ext^2_{{j}_*\Orb_{X_T}}({j}_*\shv{M},
{j}_*(\shv{M}\tensor_{\Orb_{X_T}} f_T^*I)) \cong
\Ext^2_{\Orb_{X_T}}(\shv{M},  
 \shv{M}\tensor_{\Orb_{X_T}} f_T^*I)
\]
whose vanishing is a necessary and sufficient condition for there to
exist a lift of $\shv{M}$ over $T'$. Thus, there is a natural
transformation 
\[
\mathrm{o}^2((T,\shv{M}),-) : \ker
\mathrm{o}^1((T,\shv{M}),-) \Rightarrow
\Ext^2_{\Orb_{X_T}}(\shv{M},\shv{M} \tensor_{\Orb_{X_T}} f_T^*(-))
\]
such that the pair
$\{\mathrm{o}^1((T,\shv{M}),-),\mathrm{o}^2((T,\shv{M}),-)\}$
defines a $2$-step obstruction theory for the $S$-groupoid
$\QCOHSTK[\bflat]{X}{S}$ at $(T,\shv{M})$. 

In the case where ${i} = \trvext{T}{I} : T \hookrightarrow
T\extn{I}$, the trivial $X$-extension of $T$ by $I$, then the map
$\shv{M}\tensor_{\Orb_{X_T}} f_T^*I \to \shv{M}\tensor_{\Orb_{X_T}} J$ 
is an isomorphism. By \cite[IV.3.1.12]{MR0491680}, we obtain natural
isomorphisms of abelian groups:  
\begin{align*}
  \Aut_{\QCOHSTK[\bflat]{X}{S}/S}((T,\shv{M}),I) &\cong
  \Hom_{{j}_*\Orb_{X_T}}({j}_* \shv{M},{j}_*
  (\shv{M}\tensor_{\Orb_{X_T}} f^*_TI)),  \\
  &\cong \Hom_{\Orb_{X_T}}(\shv{M},\shv{M}\tensor_{\Orb_{X_T}} f^*_TI),\\
  \Def_{\QCOHSTK[\bflat]{X}{S}/S}((T,\shv{M}),I) &\cong
  \Ext^1_{{j}_*\Orb_{X_T}}({j}_* \shv{M},{j}_*
  (\shv{M}\tensor_{\Orb_{X_T}} f^*_TI)),\\
  &\cong   \Ext^1_{\Orb_{X_T}}(\shv{M},\shv{M}\tensor_{\Orb_{X_T}}
  f^*_TI). 
\end{align*}
In \cite{hall_obstruction_thy}, using simplicial techniques, 
we will exhibit a $1$-step obstruction theory for
$\QCOHSTK[\bflat]{X}{S}$.
\begin{proof}[Proof of Theorem \ref{thm:flshvs_alg}]
  Using standard reductions \cite[App.\ B]{rydh-2009}, we are free to
  assume that $f$ is, in addition, finitely presented, and the
  scheme $S$ is affine and 
  of finite type over $\spec \Z$ (in particular, it is noetherian and 
  excellent). We now verify   
  the conditions of Theorem \ref{mainthms:repcrit2}. Certainly, the
  $S$-groupoid $\COHSTK{X}{S}$ is a limit
  preserving \'etale stack over $S$. By Lemma \ref{lem:flshvs_homog},
  we know that it is also $\HA$-homogeneous. Consider a noetherian
  local ring $(B,\mathfrak{m})$, which is $\mathfrak{m}$-adically
  complete, and a map $\spec B \to S$, then the 
  canonical functor: 
  \[
  \QCOH[\bflat,\fp,\bcompact]{X_{\spec B}} \to
  \varprojlim_n \QCOH[\bflat,\fp,\bcompact]{X_{\spec (B/\mathfrak{m}^n)}},
  \]
  is an equivalence of categories \cite[Thm.\
  1.4]{MR2183251}. Let $(T,\shv{M})$ be a
  $\COHSTK{X}{S}$-scheme, then we have determined that:
  \begin{align*}
    \Aut_{\COHSTK{X}{S}/S}((T,\shv{M}),-) &=
    \Hom_{\Orb_{X_T}}(\shv{M},\shv{M}\tensor_{\Orb_{X_T}} f_T^*(-)), \\
    \Def_{\COHSTK{X}{S}/S}((T,\shv{M}),-) &=
    \Ext^1_{\Orb_{X_T}}(\shv{M},\shv{M}\tensor_{\Orb_{X_T}}
    f_T^*(-)),\\
    \mathrm{O}^1((T,\shv{M}),-) &=
    \Hom_{\Orb_{X_T}}(   \shv{M} \underset{\Orb_{X_T}}{\tensor}
    \STor^{S,\tau,f}_1(\Orb_T,\Orb_X) , \shv{M} \underset{\Orb_{X_T}}{\tensor}
    f_T^*(-)), \\ 
    \mathrm{O}^2((T,\shv{M}),-) &=
    \Ext^2_{\Orb_{X_T}}(\shv{M},\shv{M}\tensor_{\Orb_{X_T}} f_T^*(-)),
  \end{align*}
  where $\{
  \mathrm{O}^1((T,\shv{M}),-),
  \mathrm{O}^2((T,\shv{M}),-)\}$
  are the obstruction spaces for a $2$-step obstruction theory. If $T$
  is assumed to be locally noetherian, then by Theorem \ref{thm:tor},
  the $\Orb_{X_T}$-module $\STor^{S,\tau,f}_1(\Orb_T,\Orb_X)$ is
  coherent. In addition, if $T$ is affine, \cite[Thm.\
  C]{hallj_coho_bc} implies that the functors 
  listed above are 
  coherent. Having met the conditions of Theorem
  \ref{mainthms:repcrit2}, we see that the $S$-groupoid
  $\QCOHSTK[\bflat,\fp,\bcompact]{X}{S}$ is algebraic and locally of
  finite presentation over $S$. 

  It remains to show that the diagonal of
  $\COHSTK{X}{S}$ is affine. Let
  $(T,\shv{M})$, $(T,\shv{N})$ be
  $\COHSTK{X}{S}$-schemes, then the commutative diagram in the
  category of $T$-presheaves:
  \[
  \xymatrix@C-0.85pc{\Isomstk_{\QCOHSTK{X}{S}}((T,\shv{M}),(T,\shv{N}))
    \ar[d]_{\lambda\mapsto (\lambda,\lambda^{-1})} \ar[r]
    & \Hom_T(-,T) \ar[d]^{(\ID{\shv{M}}, \ID{\shv{N}})}\\
    \Homstk_{\Orb_{X_T}/T}(\shv{M},\shv{N}) \times 
    \Homstk_{\Orb_{X_T}/T}(\shv{N},\shv{M}) \ar[r] &
    \Homstk_{\Orb_{X_T}/T}(\shv{M},\shv{M}) \times 
    \Homstk_{\Orb_{X_T}/T}(\shv{N},\shv{N}),}  
  \]
  where the morphism along the base is $(\mu,\nu) \mapsto (\nu\circ
  \mu, \mu\circ \nu)$, is cartesian. By \cite[Thm.\ D]{hallj_coho_bc}
  we deduce the result.  
\end{proof}
\section{Application II: the Hilbert stack and spaces of
  morphisms}\label{sec:example_hs} 
Fix a scheme $S$ and a $1$-morphism of algebraic stacks $f : X  
\to S$. For an $S$-scheme $T$, consider a property $P$ of a 
morphism $Z  \to X_T$. Such properties $P$ could be (but not
limited to):
\begin{enumerate}
\item[\textbf{qf}] -- quasi-finite,
\item[\textbf{lfpb}] -- the composition $Z \to X_T \to T$ is locally of
  finite presentation, 
\item[\textbf{prb}] -- the composition $Z \to X_T \to T$ is proper,
\item[\textbf{flb}] -- the composition $Z \to X_T \to T$ is flat.
\end{enumerate}
Define $\MOR[P]{X}{S}$ to be  
the category with objects pairs $(T, Z \xrightarrow{g} X_T)$, where
$T$ is an $S$-scheme and $g : Z \to X_T$ is a \emph{representable}
morphism of algebraic $S$-stacks that is $P$. Morphisms 
$(p,\pi) : (V,W \xrightarrow{h} X_V) \to (T,Z \xrightarrow{g}
X_T)$ in the category $\MOR[P]{X}{S}$ are $2$-cartesian diagrams: 
\[
\xymatrix@-0.4pc{W \ar[d]_{\pi} \ar[r]^h  & 
  X_V \ar[r]^{f_V} \ar[d]_{p_{X_T}} & V \ar[d]^p 
  \ar[d] \\ Z \ar[r]_g & X_T \ar[r]_{f_T} & T.} 
\]
If the property $P$ is reasonably well-behaved, the natural functor
$\MOR[P]{X}{S} \to \SCH{S}$ defines an $S$-groupoid. We define the
\fndefn{Hilbert stack}, $\HS{X}{S}$, to be the $S$-groupoid
$\MOR[\bflat,\blfp,\bcompact,\qf]{X}{S}$. This Hilbert stack contains
A.~Vistoli's Hilbert stack \cite{MR1138256} as well as the stack of
branchvarieties \cite{MR2608190}. We will prove the following
Theorem. 
\begin{thm}\label{thm:sep_hilbert_stack_alg}
  Fix a scheme $S$ and a morphism of algebraic stacks $f : X \to
  S$, which is separated and locally of finite presentation. Then, 
  $\HS{X}{S}$ is an 
  algebraic stack, locally of finite presentation over $S$, with
  affine diagonal over $S$.     
\end{thm}
Theorem \ref{thm:sep_hilbert_stack_alg} was the result alluded to in
the title of M. Lieblich's paper \cite{MR2233719}, though a precise
statement was not given. Theorem \ref{thm:sep_hilbert_stack_alg} was
established in \opcit using an auxillary representability result
\opcit[, Prop.\ 2.3] combined with \opcit[, Thm.\ 2.1] (Theorem
\ref{thm:flshvs_alg}). In the non-flat case, the obstruction theory
used in \opcit[, Proof of Prop.\ 2.3] is incorrect (a variant of
Example \ref{ex:non_flat_non_split} can be made into a counterexample
in this setting also). The stated obstruction theory can be made into
the second step of a 2-step obstruction theory, however. The
properties of the diagonal of $\HS{X}{S}$ have not been addressed
previously. We would like to reiterate what was stated in the
Introduction: the just mentioned errors have no net effect on the main
ideas of the articles. 
\begin{cor}\label{cor:hom_stk_alg}
  Fix a scheme $S$, and morphisms of algebraic stacks $f : X \to S$
  and $g : Y \to S$. Let $f$ be locally of finite
  presentation, proper, and flat; and $g$ locally of finite
  presentation with finite diagonal. Then, $\underline{\Hom}_S(X,Y)$
  is an algebraic stack, locally of finite 
  presentation over $S$, with affine diagonal over $S$.   
\end{cor}
Corollary \ref{cor:hom_stk_alg} can be used in the construction of the
stack of twisted stable maps \cite[Prop.\ 4.2]{MR2786662}. The
original construction of the stack of twisted stable maps utilized an
incorrect obstruction theory in the non-flat case \cite[Lem.\
5.3.3]{MR1862797}. The original proof of Corollary
\ref{cor:hom_stk_alg}, due to M. Aoki
\cite[\S3.5]{MR2194377,MR2258535}, also has an incorrect obstruction theory
in the case of a non-flat target. The stated obstruction theories, as
before, can be realized as the second step of a 2-step obstruction
theory. A variant of Example \ref{ex:non_flat_non_split} can be made
into counterexamples in these settings too. We would like to reiterate
what was stated in the Introduction and above: the just mentioned
errors have no net effect on the main ideas of the articles. 

To prove Theorem
\ref{thm:sep_hilbert_stack_alg}, we will apply Theorem
\ref{mainthms:repcrit2} directly (though as mentioned previously, this
could be done as in \cite{MR2233719} using Theorem
\ref{thm:flshvs_alg}).  With Theorem \ref{thm:sep_hilbert_stack_alg} 
proven it is easy to deduce Corollary 
\ref{cor:hom_stk_alg} via the standard method of associating to a
morphism its graph, thus the proof is omitted. Now, just as in
\S\ref{sec:example_qcoh}, there are inclusions:  
\[
\MOR[\bflat,\blfp,\bcompact,\qf]{X}{S} \subset 
\MOR[\bflat,\blfp,\bcompact]{X}{S} \subset 
\MOR[\bflat,\blfp]{X}{S} \subset
\MOR[\bflat]{X}{S}.
\]
The first two inclusions are trivially formally \'etale. By Lemma
\ref{lem:fl_nilimmers}, the third inclusion is formally \'etale.
Thus, by Lemma
\ref{lem:homog_prop}(\ref{lem:homog_prop:item:comp}\&\ref{lem:homog_prop:item:fet})
they will all be $\HA$-homogeneous if we can show that the
$S$-groupoid $\MOR[\bflat]{X}{S}$ is $\HA$-homogeneous. Also, by
Lemmata \ref{lem:def_fet} and \ref{lem:obs_fet}, descriptions of the
automorphisms, deformations, and obstructions for $\MOR[\bflat]{X}{S}$
descend to the subcategories listed above.
\begin{lem}
  Fix a scheme $S$ and a morphism of algebraic stacks $f : X \to
  S$. Then, the $S$-groupoid $\MOR[\bflat]{X}{S}$ is
  $\HA$-homogeneous. 
\end{lem}
\begin{proof}
  First we check $(\mathrm{H}_2^\HA)$. Fix a diagram of
  $\MOR[\bflat]{X}{S}$-schemes 
  \[
  [(T_1,Z_1 \xrightarrow{g_1} X_{T_1}) \xleftarrow{({i},\phi)}
  (T_0,Z_0 \xrightarrow{g_0} X_{T_0}) \xrightarrow{(p,\pi)} (T_2,Z_2
  \xrightarrow{g_2} X_{T_2}) ],
  \]
  where ${i}$ is a locally nilpotent closed immersion and $p$ is
  affine, and a 
  cocartesian square of $S$-schemes:
  \[
  \xymatrix@-0.8pc{T_0 \ar@{^(->}[r]^{i} \ar[d]_p & T_1 \ar[d]^{p'} \\ T_2
    \ar@{^(->}[r]^{{i}'} & T_3. }
  \]
  By Proposition \ref{prop:schlessinger_pushouts_new}, there exists a
  $2$-commutative diagram of algebraic $S$-stacks:
  \[
  \xymatrix@R-1.7pc{   & Z_0 \ar[dd]|!{[dl];[d]}\hole
    \ar[dl]_-{\pi} \ar[rr]|{\,\phi\,} & & 
    Z_1 \ar[dl]^{\pi'}
    \ar[dd] \\ Z_2 
    \ar[dd]  \ar[rr]|(0.64){\,\phi'\,} & & Z_3 \ar[dd]  &   \\   & T_0
    \ar[rr]|(0.36){\,{i}\,}|!{[r];[dr]}\hole
    \ar[dl]_p & & T_1 \ar[dl]^-{p'}\\ 
    T_2 \ar[rr]|{\,{i}'\,} & &T_3,  &   }
  \]
  where the left and rear faces of the cube are $2$-cartesian, and the
  top and 
  bottom faces are $2$-cocartesian in the $2$-category of algebraic
  $S$-stacks. Thus, the universal properties guarantee the existence
  of a unique $T_3$-morphism $Z_3 \xrightarrow{g_3} X_{T_3}$. By Lemma
  \ref{lem:sadm_fl_loc_SA3}, the morphism $Z_3 \to T_3$ is flat and
  all faces of the cube are $2$-cartesian. In particular, the resulting 
  $\MOR[\bflat]{X}{S}$-scheme diagram
    \[
  \xymatrix@R-1pc{(T_0,Z_0 \xrightarrow{g_0} X_{T_0}) \ar[d]_{(p,\pi)}
    \ar@{^(->}[r]^{({i},\phi)}  & 
    (T_1,Z_1 \xrightarrow{g_1} X_{T_1}) \ar[d]^{(p',\pi')} \\   (T_2,Z_2
    \xrightarrow{g_2} 
    X_{T_2}) \ar@{^(->}[r]^{({i}',\phi')} & 
    (T_3,Z_3 \xrightarrow{g_3} X_{T_3}), }
  \]
  is cocartesian in the category of $\MOR{X}{S}$-schemes. Condition
  $(\mathrm{H}_1^\HA)$ follows from a similar argument as that given
  in Lemma \ref{lem:flshvs_homog}.   
\end{proof}
Fix a $\MOR[\bflat]{X}{S}$-scheme $(T,Z \xrightarrow{g} X_T)$ and a
quasicoherent $\Orb_T$-module $I$. Then, unravelling the definitions
and applying the results of \cite{MR2206635}, demonstrates that there
are natural isomorphisms of abelian groups: 
\begin{align*}
  \Aut_{\MOR[\bflat]{X}{S}/S}((T,Z \xrightarrow{g} X_T),I) &\cong
  \Hom_{\Orb_Z}(L_{Z/X_T},g^*f_T^*I)\\
  \Def_{\MOR[\bflat]{X}{S}/S}((T,Z \xrightarrow{g} X_T),I) &\cong
  \Ext_{\Orb_Z}^1(L_{Z/X_T},g^*f_T^*I).  
\end{align*}
Using identical ideas to those developed in \S\ref{sec:example_qcoh},
together with \opcit, we obtain a $2$-term obstruction theory  for the
$S$-groupoid $\MOR[\bflat]{X}{S}$ at $(T,Z \xrightarrow{g} X_T)$:  
\begin{align*}
  \mathrm{o}^1((T,Z \xrightarrow{g} X_T),-)
  &: \Exal_S(T,-) \Rightarrow  
  \Hom_{\Orb_{Z}}(g^*\STor^{S,\tau,f}_1(\Orb_T,\Orb_X),g^*f_T^*(-))\\
  \mathrm{o}^2((T,Z \xrightarrow{g} X_T),-) &: \ker \mathrm{o}^1((T,Z
  \xrightarrow{g} X_T),-)\Rightarrow \Ext^2_{\Orb_Z}
  (L_{Z/X_T},g^*f_T^*(-)). 
\end{align*}
In \cite{hall_obstruction_thy}, using simplicial techniques, 
we will exhibit a $1$-step obstruction theory for
$\MOR[\bflat]{X}{S}$.
\begin{proof}[Proof of Theorem \ref{thm:sep_hilbert_stack_alg}]
  The proof that the $S$-groupoid $\HS{X}{S}$ is algebraic and locally
  of finite presentation is essentially identical to the proof of
  Theorem \ref{thm:flshvs_alg}, thus is omitted. It remains to show
  that the diagonal is affine. So, let $(T,Z_1\xrightarrow{g_1} X_T)$
  and $(T,Z_2\xrightarrow{g_2} X_T)$ be $\HS{X}{S}$-schemes, then
  the inclusion of $T$-presheaves: {\small\[
    \Isomstk_{\HS{X}{S}}((T,Z_1\xrightarrow{g_1} X_T),
    (T,Z_2\xrightarrow{g_2} X_T)) \subset
    \Isomstk_{\QCOHSTK{X}{S}}((T,(g_2)_*\Orb_{Z_2}),(T,(g_1)_*\Orb_{Z_1})),
  \]}
  \hspace{-1.8mm} is representable by closed immersions. By
  Theorem \ref{thm:flshvs_alg}, we deduce the result. 
\end{proof}
\appendix
\section{Homogeneity of stacks}\label{app:schlessinger_pushouts} 
The results of this section are routine bootstrapping arguments. They
are included so that $\HA$-homogeneity can be proved for moduli
problems involving stacks.
\begin{defn}
  Fix a $2$-commutative diagram of algebraic stacks 
  \[
  \xymatrix{X_0 \ar@{^(->}[r]^{{i}} \ar[d]_f
    \drtwocell<\omit>{^\alpha} & X_1 \ar[d]^{f'}\\ X_2
    \ar@{^(->}[r]_{{i}'} & X_3,}
  \]
  where ${i}$ and ${i}'$ are closed immersions and $f$ and $f'$
  are affine. If the induced map:
  \[
  \Orb_{X_3} \to {i}'_*\Orb_{X_2} \times_{({i}'f)_*\Orb_{X_0}}
  f'_*\Orb_{X_1}
  \]
  is an isomorphism of sheaves, then we say that the diagram is a 
  \fndefn{geometric pushout}, and that $X_3$
  is a \fndefn{geometric pushout} of the diagram $[X_2 \xleftarrow{f}
  X_0 \xrightarrow{{i}} X_1]$. 
\end{defn}
The main result of this section is the following
\begin{prop}\label{prop:schlessinger_pushouts_new}
  Any diagram of algebraic stacks $[X_2 \xleftarrow{f}
  X_0 \xrightarrow{{i}} X_1]$, where ${i}$ is a
  locally nilpotent closed immersion and $f$ is affine, admits a
  geometric pushout $X_3$. The 
  resulting geometric pushout diagram is $2$-cartesian and
  $2$-cocartesian in the $2$-category of algebraic stacks.  
\end{prop}
We now need to collect some results which aid with the
bootstrapping process. 
\begin{lem}\label{lem:geom_po_ferr_new}
  Fix a $2$-commutative diagram of algebraic stacks:
    \[
  \xymatrix@-0.8pc{X_0 \ar@{^(->}[r]^{{i}} \ar[d]_f
    \drtwocell<\omit>{^} & X_1 \ar[d]^{f'}\\ X_2
    \ar@{^(->}[r]_{{i}'} & X_3}
  \]
  \begin{enumerate}
  \item\label{lem:geom_po_ferr_new:item:2cart} If the diagram is a
    geometric pushout diagram, then it is $2$-cartesian. 
  \item\label{lem:geom_po_ferr_new:item:flb_gpo} If the diagram above
    is a geometric pushout diagram, then it 
    remains so after flat base change on $X_3$.
  \item\label{lem:geom_po_ferr_new:item:fl_loc_gpo} If after
    fppf base change on $X_3$, the above diagram  
    is a geometric pushout diagram, then it was a geometric pushout prior to
    base change.
  \end{enumerate}
 \end{lem}
\begin{proof}
  The claim
  \itemref{lem:geom_po_ferr_new:item:2cart} is local on $X_3$ for the
  smooth topology, thus we may assume that everything in sight is
  affine---whence the result follows from \cite[Thm.\ 2.2]{MR2044495}.
  Claims \itemref{lem:geom_po_ferr_new:item:flb_gpo} and
  \itemref{lem:geom_po_ferr_new:item:fl_loc_gpo} are trivial
  applications of flat descent. 
\end{proof}
\begin{lem}\label{lem:sadm_fl_loc_SA3}
  Consider a $2$-commutative diagram of algebraic stacks: 
  \[
  \xymatrix@R-1.6pc@C-.5pc{   & D_0 \ar[dd]|!{[dl];[d]}\hole \ar[dl]
    \ar@{^(->}[rr] & & 
    D_1 \ar[dl] 
    \ar[dd] \\ D_2
    \ar[dd]  \ar@{^(->}[rr] & & D_3 \ar[dd]  &   \\   & C_0
    \ar@{^(->}[rr]|!{[r];[dr]}\hole 
    \ar[dl] & & C_1 \ar[dl]\\ 
    C_2 \ar@{^(->}[rr] & &C_3  &   } 
  \]
  where the back and left faces of the cube are $2$-cartesian, the
  top and bottom faces are geometric pushout diagrams, and for $i=0$, $1$,
  $2$, the morphisms $D_i \to C_i$ are flat. Then, all
  faces of the cube are $2$-cartesian and the morphism $D_3 \to C_3$
  is flat. 
\end{lem}
\begin{proof}
  By Lemma
  \ref{lem:geom_po_ferr_new}\itemref{lem:geom_po_ferr_new:item:flb_gpo},
  this is all smooth local on $C_3$ and $D_3$, thus we immediately 
  reduce to the case where everything in sight is affine. Fix a
  diagram of rings $[A_2 \rightarrow A_0 \xleftarrow{p} A_1]$ where $p
  : A_1\to A_0$ is surjective. For $i=0$, $1$, $2$ fix flat
  $A_i$-algebras $B_i$, and $A_0$-isomorphisms
  $B_2\tensor_{A_2} A_0 \cong B_0$ and $B_1\tensor_{A_1} A_0 \cong
  B_0$. Set $A_3 = A_2\times_{A_0} A_1$ and $B_3 = B_2\times_{B_0} 
  B_1$, then we have to prove that the $A_3$-algebra $B_3$ is
  flat, the natural maps $B_3\tensor_{A_3} A_i \to B_i$ are
  isomorphisms, and that these isomorphisms are compatible with the
  given isomorphisms. This is an immediate consequence of \cite[Thm.\
  2.2]{MR2044495}, since these are questions about modules.   
\end{proof}
We omit the proof of the following easy result from commutative
algebra. 
\begin{lem}\label{lem:nilp_fp}
  Fix a surjection of rings $A \to A_0$ and let $I=\ker(A\to
  A_0)$. Suppose that there is a $k$ such that $I^k = 0$.
  \begin{enumerate}
  \item \label{item:nilp_surj} Given a map of $A$-modules $u : M
    \to N$ such that $u\tensor_A A_0$ is surjective,  then $u$ is
    surjective.  
  \item \label{item:nilp_mod} For an $A$-module $M$, if $M\tensor_A
    A_0$ is finitely generated, then $M$ is finitely generated.
  \item Given an $A$-algebra $B$ and a $B$-module $M$, let $M_0 =
    A_0\tensor_A M$.  
    \begin{enumerate}
    \item \label{item:nilp_fpmod} If $M$ is $A$-flat and $M_0$ is
      $B_0$-finitely presented, then $M$ is $B$-finitely presented.      
    \item \label{item:nilp_ft} If $B_0$ is a finite type
      $A_0$-algebra, then $B$ is a  finite type $A$-algebra. 
    \item \label{item:nilp_fpf} If $B$ is a flat $A$-algebra and $B_0$ is
      a finitely presented $A_0$-algebra, then $B$ is a finitely
      presented $A$-algebra.
    \end{enumerate}
  \end{enumerate}
\end{lem}
\begin{lem}\label{lem:fl_nilimmers}
  Fix a morphism $f : X \to Y$ of algebraic stacks and a locally
  nilpotent closed immersion $Y_0\hookrightarrow Y$. If 
  $f$ is flat, then it is locally of finite presentation
  (resp.\ smooth) if and only if the same is true of the map $X\times_Y
  Y_0 \to Y_0$.   
\end{lem}
\begin{proof}
  Observe that for flat
  morphisms which are locally of finite presentation, smoothness is a
  fibral condition, thus follows from the first claim. The first claim is
  smooth local on $Y$ and $X$, thus follows from Lemma
  \ref{lem:nilp_fp}\itemref{item:nilp_fpf}. 
\end{proof}
\begin{lem}\label{lem:lift_pres}
  Consider a locally nilpotent closed immersion of algebraic stacks $X
  \hookrightarrow X'$ and a smooth morphism $U \to X$ where $U$ is an
  affine scheme. Then, there exists a smooth morphism $U' \to X'$
  which pulls back to $U \to X$.   
\end{lem}
\begin{proof}
  Since $U$ is quasicompact, it is sufficient to treat the case where
  the locally nilpotent closed immersion $X \hookrightarrow X'$ is
  square zero. Then, \cite[Thm.\ 
  1.4]{MR2206635} implies that the obstruction to the existence of
  a flat lift lies in the abelian group
  $\Ext^2_{\Orb_U}(L_{U/X},M)$, for some quasicoherent
  $\Orb_U$-module $M$. The morphism $U \to X$ is smooth, $U$ is
  affine, and the $\Orb_U$-module $\SHom_{\Orb_U}(\Omega_{U/X},M)$ is
  quasicoherent, thus $\Ext^2_{\Orb_U}(L_{U/X},M) =
  H^2(U,\SHom_{\Orb_U}(\Omega_{U/X},M)) = 0$. Finally,
  by Lemma \ref{lem:fl_nilimmers}, any such lift that is flat, is also
  smooth.  
\end{proof}
\begin{proof}[Proof of Proposition \ref{prop:schlessinger_pushouts_new}]  
  Throughout, the following notation will be used.
  \begin{enumerate}
  \item[(i)] For ${d}=1$, $2$, $3$, let $\Coll_{d}$ denote the full
    $2$-subcategory of the $2$-category of algebraic stacks, with
    objects those algebraic stacks whose ${d}$th diagonal is
    affine. Note that $\Coll_3$ is the full $2$-category of algebraic
    stacks. 
  \item[(ii)] Let $\Coll_0$ denote the category of affine schemes.
  \item[(iii)] Fix an algebraic stack $Y$ and a collection of
    morphisms $\{Y^l \to Y\}_{l\in \Lambda}$. For $i$, $j$, $k\in \Lambda$, set
    $Y^{ij}$ (resp.\ $Y^{ijk}$) to be $Y^i\times_Y Y^j$
    (resp. $Y^i\times_Y Y^j\times_Y Y^k$).
  \end{enumerate}  
  \emph{Claim:} Let ${d}=0$, $1$, or $2$. Suppose that any diagram
  $[X_2\xleftarrow{f} X_0 \xrightarrow{{i}} X_1]$ in $\Coll_{d}$, where
  ${i}$ is a locally nilpotent closed immersion and $f$ is affine,
  admits a geometric pushout in $\Coll_{d}$, such that the
  resulting geometric pushout diagram is $2$-cartesian and
  $2$-cocartesian in the $2$-category $\Coll_{d}$, then the same is true
  of $\Coll_{{d}+1}$.
 
  To see that the Claim is sufficient to prove the Proposition,
  observe that given a diagram $[X_2\xleftarrow{f} X_0
  \xrightarrow{{i}} X_1]$ in $\Coll_0$, where ${i}$ is a locally 
  nilpotent closed immersion and $f$ is affine, the affine scheme
  $X_3=\spec (\Gamma(\Orb_{X_2}) \times_{\Gamma(\Orb_{X_0})}
  \Gamma(\Orb_{X_1}))$ is a geometric pushout of this diagram. The
  resulting geometric pushout diagram is trivially seen to be
  cocartesian in $\Coll_0$; by Lemma
  \ref{lem:geom_po_ferr_new}\itemref{lem:geom_po_ferr_new:item:2cart},
  it is also $2$-cartesian. By induction, the Claim now implies that
  every such diagram in $\Coll_3$ admits a geometric pushout, and the
  resulting geometric pushout diagram is $2$-cocartesian and
  $2$-cartesian in $\Coll_3$. Since every algebraic stack belongs to 
  $\Coll_3$, we deduce the Proposition.\\
  \emph{Proof of Claim:} 
  First, we show that geometric pushout diagrams in $\Coll_{{d}+1}$ are
  $2$-cocartesian (they are always $2$-cartesian by Lemma
  \ref{lem:geom_po_ferr_new}\itemref{lem:geom_po_ferr_new:item:2cart}). Thus,
  we must uniquely complete all $2$-commutative diagrams: 
  \vspace{-8mm}
  \[
  \xymatrix@ur@-.8pc{X_0 \ar@{^(->}[r]^{{i}} \ar[d]_f
    \drtwocell<\omit>{\alpha} & X_1 
    \ar[d]^{f'} \ar@/^1pc/[ddr]^{\psi_1} & \\  X_2
    \ar@{^(->}[r]_{{i}'} \ar@/_1pc/[drr]_{\psi_2} & X_3
    \drtwocell<\omit>{\beta} & \\ & & W,}    
  \vspace{-8mm}
  \]  
  where the square is a geometric pushout diagram in $\Coll_{{d}+1}$,
  ${i}$ is a locally nilpotent closed immersion, and $W \in
  \Coll_{{d}+1}$. Observe that if $\amalg_{l\in \Lambda} X_3^l \to
  X_3$ is a smooth cover of $X_3$, where each $X_3^l$ is affine, then
  $\forall i$, $j$, $k\in \Lambda$, we have that $X_3^{ij}$, $X_3^{ijk}
  \in \Coll_{d}$. By smooth descent, it is thus sufficient to prove
  that the diagram above is $2$-cocartesian when the $X_i\in
  \Coll_d$. If $d\neq 0$, we may repeat this argument 
  $(d-1)$-more times, to reduce to the case where $X_i\in
  \Coll_0$. That is, they are affine. 

  Now we show the uniqueness of completions of the diagram. Suppose
  that for $j=1$, $2$ we have a 
  $1$-morphism $g^j : X_3 \to W$ together with $2$-morphisms
  ${\gamma_1^j}: \psi_1 
  {\Rightarrow} g^j\circ f'$ and ${\gamma_2^j}  : g^j\circ {i}'  
  {\Rightarrow} \psi_2$, satisfying $(f^*\gamma_2^j)\circ \alpha
  \circ ({i}^*\gamma_1^j) = \beta$. We claim that there is a unique
  $2$-morphism $\eta : g^1 \Rightarrow g^2$ such that $\eta\circ
  \gamma^1_1 = \gamma_1^2$ and $\gamma_2^2\circ \eta = \gamma_2^1$. To see
  this, let $e : E = X_3\times_{(g^1,g^2),W\times W, \Delta} W \to X_3$
  be the equalizer of the pair of $1$-morphisms $g^1$ and $g^2$. Since
  $X_3$ affine and $W \in \Coll_{{d}+1}$, we have that $E \in
  \Coll_{d}$. By the universal property defining $E$, we obtain 
  $1$-morphisms $\tilde{f}' :   X_1 \to  E$ and $\tilde{{i}}' : X_2
  \to E$ such that $e\circ \tilde{f}' = f'$ and $e\circ
  \tilde{{i}}' = {i}'$, which are unique up to a unique
  $2$-morphism. Since our square is a geometric pushout diagram, by
  assumption it is $2$-cartesian in $\Coll_{d}$. Thus, there exists a
  $1$-morphism $s : X_3 \to  E$, which is unique up to a unique
  $2$-morphism, that is compatible with this data. Since $e\circ s =
  \ID{X_3}$, the definition of $E$ gives a 
  unique $2$-morphism $\eta : g^1 \Rightarrow g^2$ with the required
  properties. 

  Now we show the existence of a completion of the diagram. Fix a
  smooth presentation $\amalg_{l\in \Lambda} W^l \to W$, where each
  $W^l$ is an affine scheme. For $m\neq 3$ and $l\in \Lambda$ set
  $X_m^l=X_m\times_W W^l$. Since $W\in
  \Coll_{{d}+1}$ and the schemes $X_m$ are affine, for $m\neq 3$ the
  stacks $X^l_m$ all belong to $\Coll_{d}$. Thus, the geometric
  pushout $X_3^l$ of the diagram $[X_2^l  \leftarrow X_0^l 
  \rightarrow X_1^l]$ exists, and the resulting geometric pushout diagram is
  $2$-cocartesian in $\Coll_{d}$. In particular,
  there is a unique map $X_3^l \to W^l$ which is compatible with the
  data. Similarly, there is also a unique map $X_3^l \to X_3$ which is
  compatible with the data---by Lemmata \ref{lem:sadm_fl_loc_SA3} and 
  \ref{lem:fl_nilimmers} this map is smooth. By the uniqueness
  statement that we have already proven, we obtain a unique map
  $\amalg_{l\in \Lambda} X_3^l \to W$ which is compatible with the 
  data. Since the morphism $\amalg_{l\in \Lambda} X_3^l \to X_3$ is smooth and
  surjective, smooth descent gives a map $X_3 \to W$ completing the
  diagram.   
  
  Finally, we show that any diagram $[X_2 \xleftarrow{f} X_0
  \xrightarrow{{i}} X_1]$ in $\Coll_{{d}+1}$, where ${i}$ is a locally
  nilpotent closed immersion and $f$ is affine, admits a geometric
  pushout. Fix a smooth surjection $\amalg_{l\in \Lambda} X_2^l \to
  X_2$, where $X_2^l$ is an affine scheme $\forall l\in \Lambda$. Set  
  $X_0^l = X_2^l \times_{X_2} X_0$, then as $f$ is affine,
  the scheme $X_0^l$ also affine. By Lemma \ref{lem:lift_pres}, the
  resulting smooth morphism $X_0^l \to X_0$ 
  lifts to a smooth morphism $X^l_1 \to X_1$, with $X^l_1$ affine, and
  $X_0^l\cong X^l_1\times_{X_1} X_0$. As before, $\forall m\neq 3$ and
  $\forall i$, $j$, $k\in \Lambda$ we have $X_m^{ij}$, 
  $X_m^{ijk} \in \Coll_d$. Thus, for $I=i$, $ij$ or $ijk$, a geometric
  pushout $X_3^I$ of the diagram
  $[X_2^I \leftarrow X_0^I \rightarrow X_1^I]$  
  exists, and belongs to $\Coll_d$. We have already shown that
  geometric pushouts in $\Coll_{d+1}$ are $2$-cartesian in
  $\Coll_{d+1}$, thus there are uniquely induced morphisms $X^{ij}_m
  \to X^{i}_m$. For $m\neq 3$, these morphisms are clearly smooth,
  and by Lemmata \ref{lem:sadm_fl_loc_SA3} and \ref{lem:fl_nilimmers}
  the morphisms $X^{ij}_3 \to X^i_3$ are smooth. It easily verified
  that the universal properties give rise to a smooth groupoid
  $[\amalg_{i,j\in \Lambda} X^{ij}_3 \rightrightarrows 
  \amalg_{k\in \Lambda} X_3^k]$. The quotient $X_3$ of this groupoid
  in the category of stacks is algebraic. By Lemma,
  \ref{lem:geom_po_ferr_new}\itemref{lem:geom_po_ferr_new:item:fl_loc_gpo}
  it is also a geometric pushout of the diagram $[X_2 \leftarrow X_0
  \rightarrow X_1]$.
\end{proof}
\section{Local Tor functors on algebraic stacks}\label{app:tor}
The aim of the section is to state some easy generalizations of
\cite[{III}.6.5]{EGA} to algebraic stacks. We omit the proofs
as they are simple descent arguments. 
\begin{thm}\label{thm:tor}
  Fix a scheme $S$ and a $2$-cartesian diagram of algebraic
  $S$-stacks:
  \[
  \xymatrix@-0.8pc{X_3 \ar[r]^-{f_1'} \ar[d]_{f_2'} & X_2 \ar[d]^{f_2}\\ X_1
    \ar[r]^{f_1} & X_0.}
  \]
  Then, for each integer $i\geq 0$, there exists a
  natural bifunctor: 
  \[
  \STor^{X_0,f_1,f_2}_i( - , - ) : \QCOH{X_1} \times
  \QCOH{X_2} \to \QCOH{X_3},
  \]
  The family of bifunctors $\{\STor_i^{X_0,f_1,f_2}(-,-)\}_{i\geq 0}$
  forms a $\partial$-functor in each variable. Moreover,
  there is a natural isomorphism for all $M\in \QCOH{X_1}$ and $N\in
  \QCOH{X_2}$: 
  \[
  \STor^{X_0,f_1,f_2}_0(M,N) \cong f_2'^*M \tensor_{\Orb_{X_3}} f_1'^*N.
  \]
  If $M$ or $N$ is $X_0$-flat, then for all
  $i>0$ we have $\STor_i^{X_0,f_1,f_2}(M,N) = 0$. 
  In addition, if the 
  algebraic stacks $X_1$ and $X_0$ are locally noetherian and the morphism
  $f_2$ is locally of finite type, then the bifunctor above restricts to
  a bifunctor:
  \[
  \STor^{X_0,f_1,f_2}_i( - , - ) : \COH{X_1} \times
 \COH{X_2} \to \COH{X_3}.
  \]
\end{thm}
Another result that will be useful is the following.
\begin{lem}\label{lem:aff_tor}
  Fix a scheme $S$ and a $2$-cartesian diagram of algebraic $S$-stacks
  \[
  \xymatrix@-0.8pc{W\times_Z Y \ar[d]_{g_W} \ar[r]^{h'} & X\times_Z Y
    \ar[r]^-{f_Y} \ar[d]_{g_X} & Y \ar[d]^g \\W \ar[r]^h & X \ar[r]^f & Z, }
  \]
  where the morphism $h$ is affine. Then, for any $M \in \QCOH{W}$, 
  $N\in \QCOH{Y}$, and $i\geq 0$, there is a natural isomorphism of
  quasicoherent $\Orb_{X\times_Z Y}$-modules:
  \[
  \STor^{Z,f,g}_i(h_*M,N) \cong h'_*\STor^{Z,f\circ h,g}_i(M,N).
  \]
\end{lem}
\providecommand{\bysame}{\leavevmode\hbox to3em{\hrulefill}\thinspace}
\providecommand{\MR}{\relax\ifhmode\unskip\space\fi MR }
\providecommand{\MRhref}[2]{%
  \href{http://www.ams.org/mathscinet-getitem?mr=#1}{#2}
}
\providecommand{\href}[2]{#2}

\end{document}